\documentclass{article}

\usepackage{amsmath}
\usepackage{amsfonts}
\usepackage{amssymb}
\usepackage{mathrsfs}
\usepackage[mathscr]{eucal}
\numberwithin{equation}{section}
\usepackage{amsthm}
\newtheorem{theorem}{Theorem}

\usepackage{setspace}
\usepackage[a4paper,left=3.18cm, right=3.18cm, top=2.54cm, bottom=2.54cm]{geometry} 
\usepackage{enumitem}
\setlist[enumerate]{leftmargin=.5in}
\setlist[itemize]{leftmargin=.5in}

\usepackage{xcolor}
\definecolor{trueblue}{rgb}{0.0, 0.45, 0.81}
\definecolor{lightgoldenrodyellow}{rgb}{0.98, 0.98, 0.82}
\definecolor{camel}{rgb}{0.76, 0.6, 0.42}
\definecolor{cadetblue}{rgb}{0.37, 0.62, 0.63}
\definecolor{chocolate}{rgb}{0.48, 0.25, 0.0}
\definecolor{darktaupe}{rgb}{0.28, 0.24, 0.2}

\usepackage{tikz}
\usetikzlibrary{shapes,arrows}
\usetikzlibrary{calc,patterns,angles,quotes}
\usetikzlibrary{matrix}
\usepackage{tkz-euclide}
\usetikzlibrary{tikzmark}

\usetikzlibrary{3d}
\usetikzlibrary{perspective}
\usetikzlibrary{quotes,angles}

\usepackage{tikz-3dplot} 
\usetikzlibrary{shapes,arrows}
\usetikzlibrary{positioning}

\usepackage{graphicx} 
\graphicspath{{images/}{./}}
\usepackage{multirow}  
\usepackage{subcaption}
\usepackage{epstopdf}
\usepackage{overpic}
\newlength{\Oldarrayrulewidth}
\newcommand{\Cline}[2]{
  \noalign{\global\setlength{\Oldarrayrulewidth}{\arrayrulewidth}}
  \noalign{\global\setlength{\arrayrulewidth}{#1}}\cline{#2}
  \noalign{\global\setlength{\arrayrulewidth}{\Oldarrayrulewidth}}}
\usepackage{algorithm}
\usepackage{algpseudocode}
\usepackage{rotating}
\usepackage{ctable}
\usepackage{tabu}

\usepackage[colorlinks=true, allcolors=trueblue]{hyperref}

\usepackage{titlesec}
\definecolor{MSBlue}{rgb}{.204,.353,.541}
\titleformat*{\section}{\Large\bfseries\sffamily\color{black}}
\titleformat*{\subsection}{\large\bfseries\sffamily\color{black}}

\renewenvironment{abstract}{
  \quotation
  \textbf{\textsf{\color{black}{\abstractname.}}}
  \sffamily
}{\endquotation}

\usepackage{authblk}
\newcommand*\samethanks[1][\value{footnote}]{\footnotemark[#1]}
\makeatletter
\def\@fnsymbol#1{\ensuremath{\ifcase#1\or \dagger\or \ddagger\or
   \mathsection\or \mathparagraph\or \|\or **\or \dagger\dagger
   \or \ddagger\ddagger \else\@ctrerr\fi}}
\makeatletter

\def\0{\mathbf{0}}
\def\1{\mathbf{1}}

\def\f{\mathbf{f}}

\def\g{\mathbf{g}}
\def\h{\mathbf{h}}

\def\p{\mathbf{p}}
\def\tp{\mathtt{p}}
\def\r{\mathbf{r}}

\def\x{\mathbf{x}}
\def\y{\mathbf{y}}
\def\z{\mathbf{z}}

\def\rS{\mathrm{S}}
\def\rA{\mathrm{A}}

\def\R{\mathbb{R}}
\def\cR{\mathcal{R}}
\def\A{\mathcal{A}}

\def\E{\mathcal{E}}
\def\tE{\mathtt{E}}
\def\F{\mathcal{F}}

\def\tF{\mathtt{F}}
\def\tG{\mathtt{G}}

\def\tH{\mathtt{H}}
\def\tL{\mathtt{L}}
\def\M{\mathcal{M}}
\def\N{\mathcal{N}}

\def\O{\mathrm{O}}
\def\Q{\mathcal{Q}}
\def\tR{\mathtt{R}}

\def\T{\mathcal{T}}
\def\V{\mathcal{V}}

\def\I{\mathtt{I}}
\def\B{\mathtt{B}}

\def\d{\mathrm{d}}
\def\bd{\mathbf{d}}

\newcommand{\argmin}{\operatornamewithlimits{argmin}}
\newcommand{\argmax}{\operatornamewithlimits{argmax}}
\newcommand{\mean}{\operatornamewithlimits{mean}}
\newcommand{\SD}{\operatornamewithlimits{SD}}

\title{ 
\bfseries\sffamily\color{black}{
Square-Domain Area-Preserving Parameterization for Genus-Zero and Genus-One Closed Surfaces}
}
\author{
\sffamily\color{black}
{Shu-Yung Liu\thanks{\footnotesize Department of Mathematics, National Taiwan Normal University, Taipei, Taiwan (\href{mailto:lii227857@gmail.com}{lii227857@gmail.com}, \href{mailto:yue@ntnu.edu.tw}{yue@ntnu.edu.tw}) }
\, and \, Mei-Heng Yueh\samethanks[1] }  
}

\date{}

\begin{document}
\captionsetup[figure]{labelfont={bf,sf},name={Figure},labelsep=period}
\captionsetup[table]{labelfont={bf,sf},name={Table},labelsep=period}
\maketitle

\begin{abstract}
The parameterization of closed surfaces typically requires either multiple charts or a non-planar domain to achieve a seamless global mapping. In this paper, we propose a numerical framework for the seamless parameterization of genus-zero and genus-one closed simplicial surfaces onto a unit square domain. The process begins by slicing the surface with either the shortest-path or the Reeb graph method. The sliced surface is then mapped onto the unit square using a globally convergent algorithm that minimizes the weighted variance of per-triangle area ratios to achieve area preservation. Numerical experiments on benchmark models demonstrate that our method achieves high accuracy and efficiency. Furthermore, the proposed method enables applications such as geometry images, producing accurate and high-quality surface reconstructions.


\end{abstract}

\section{Introduction}
Surface parameterization aims to map a given surface onto a canonical domain while preserving certain geometric properties \cite{FlHo05, ShPR06}. Such mappings facilitate a wide range of mesh processing tasks, including re-meshing, texture mapping, and surface registration \cite{LaLu14,LuLY14,YoMY14,YuLW17}. Ideally, a length-preserving (isometric) parameterization would be most desirable, as it preserves the intrinsic geometric information of the surface. However, achieving isometry is generally impossible because the Gaussian curvature of a surface and that of its target domain may differ. As a result, most research instead focused on either angle-preserving (conformal) or area-preserving (authalic) parameterizations, depending on the requirements of the downstream application.

Authalic parameterization preserves the local area of a surface onto a parameter domain, which is particularly valuable in applications requiring uniform vertex distribution. A notable example is spherical area-preserving mapping, which enables shape description via spherical harmonics \cite{BrGK95}. This approach has found wide use in medical anatomical analysis of structures such as the brain ventricles \cite{GeSJ01} and the hippocampus \cite{StLP04, GeCC09, ChLS20}, and more recently extended to surfaces with spherical-like shapes \cite{GiCK21, ChGK22, ChSh25}. To realize area preservation on closed surfaces, a variety of computational methods have been proposed, including optimal transport formulations \cite{NaSZ17, CuQW19}, density-equalizing methods \cite{ChGK22, LyLC24, Choi24, YaCh26}, and energy minimization methods \cite{YuLL19, SuYu24, LiYu25b, SuYu24}.

Area-preserving parameterization is also crucial for open surfaces, especially in the geometry image representation \cite{GuGH02}. A geometry image encodes a surface into a regular RGB image, enabling the compression of a surface with a large number of vertices into a compact image. In this context, area preservation is essential for accurate surface reconstruction from the geometry image, as shown in \cite{LiYu25}. A variety of methods have been developed, including stretch-minimizing methods \cite{YoBS04}, optimal transport maps \cite{ZhSG13}, density-equalizing maps \cite{ChRy18, ChCR20, GiCK21, ChGK22, ChSh25}, and energy minimization methods \cite{YuLW19, LiYu24}.

Geometry images can represent closed surfaces by slicing the mesh along appropriate cutting paths to obtain a disk-type surface, and then parameterizing the sliced surface onto a square domain. In this paper, we focus on genus-zero and genus-one surfaces. With our proposed slicing strategy, these surfaces are sliced into simply connected ones that can be parameterized onto the unit-square domain. We then apply authalic energy minimization \cite{LiYu24} to compute area-preserving parameterizations, which have been shown to achieve higher accuracy than state-of-the-art methods. Finally, we demonstrate the application of our method to the geometry image representation of surfaces. By introducing a specific area measure and applying angular correction to reduce severe distortions, our framework enables high-quality surface reconstructions from geometry images.

\subsection{Contribution}
The main contributions of this work are as follows:
\begin{itemize}
\item We introduce a slicing algorithm that cuts genus-zero and genus-one closed surfaces along a prescribed path, producing a simply connected open surface.
\item We prove that the authalic energy is equivalent to the weighted variance of per-triangle area ratios, and formulate the square-domain area-preserving parameterization problem as an authalic energy minimization (AEM) problem.
\item We develop a nonlinear optimization framework with guaranteed convergence, initializing the map via a sequential quadratic approximation.
\item We demonstrate an application of our method to geometry images for effective surface representation.
\end{itemize}

\subsection{Organization of the paper}
The remainder of this paper is organized as follows. In Section \ref{sec:2}, we review preliminaries on simplicial surfaces and mappings. In Section \ref{sec:3}, we introduce a slicing method that produces simply connected open surfaces for square-domain parameterization. Section \ref{sec:4} introduces the authalic energy functional and establishes its equivalence to the weighted variance of area ratios. In Section \ref{sec:5}, we present a numerical algorithm combining fixed-point initialization with a preconditioned nonlinear conjugate gradient method for the AEM problem. Section \ref{sec:6} demonstrates the associated numerical experiments. In Section \ref{sec:7}, we apply the method to a geometry image and provide the framework to improve surface reconstruction. Section \ref{sec:8} discusses connections with prior energy formulations and the role of area preservation in geometry images. Section \ref{sec:9} concludes the paper.

\subsection{Notation}
In this paper, unless otherwise noted, the notation follows the rules as:
\begin{itemize}[label={$\bullet$}]
\item Real-valued vectors are represented using bold letters, such as $\mathbf{f}$.
\item Real-valued matrices are represented using capital letters, such as $L$.
\item Ordered sets of indices are denoted using typewriter letters, such as $\mathtt{I}$ and $\mathtt{B}$.
\item The $i$-th component of a vector $\mathbf{f}$ is written as $\mathbf{f}_i$. For an index set $\mathtt{I}$, the corresponding subvector of $\mathbf{f}$, consisting of entries $\mathbf{f}_i$ with $i \in \mathtt{I}$, is denoted by $\mathbf{f}_\mathtt{I}$.
\item The $(i,j)$-th entry of a matrix $L$ is written as $L_{i,j}$. For index sets $\mathtt{I}$ and $\mathtt{J}$, the submatrix of $L$, consisting of entries $L_{i,j}$ with $i \in \mathtt{I}$ and $j \in \mathtt{J}$, is denoted by $L_{\mathtt{I},\mathtt{J}}$.
\item The set of real numbers is represented by $\mathbb{R}$.
\end{itemize}

\section{Simplicial manifolds and simplicial mappings} \label{sec:2}
A \textit{triangular mesh} is defined as $\T = \big(\V, \E, \F \big)$, where $\V$, $\F$, and $\E$ are the sets of $n$ vertices, $m$ oriented triangular faces, and undirected edges, respectively:
\begin{subequations} \label{eq:mesh}
\begin{align}
\V &= \left\{ v_\ell = (v_\ell^1, v_\ell^2, v_\ell^3) \in\R^3 \right\}_{\ell=1}^n, \\[4pt]
\F &= \left\{ \tau_s = [ {v}_{i_s}, {v}_{j_s}, {v}_{k_s} ] \subset\R^3 \mid {v}_{i_s}, {v}_{j_s}, {v}_{k_s}\in\V \right\}_{s=1}^m, \\[4pt]
\E &= \left\{ [v_i,v_j] \subset\R^3 \mid [v_i,v_j,v_k]\in\F  \right\}.
\end{align}
\end{subequations}
Here, $[{v}_{i}, {v}_{j}, {v}_{k}]$ denotes the oriented \textit{$2$-simplex}, i.e., the triangle formed by the vertices $v_i, v_j, v_k$.
A \textit{simplicial surface}~$\M$ is the \textit{geometric realization} of a triangular mesh $\T$, defined as the union of flat triangles in $\R^3$, i.e., $\M = \cup_{\tau\in\F}~ \tau$. Each triangle is determined by the positions of the three vertices associated with a face in the set $\mathcal{F}$.

A {\it simplicial mapping} $f: \M \to \mathbb{R}^2$ is a piecewise affine mapping that is completely determined by its values on the vertices, denoted by
$$
f_i \equiv f(v_i) = (f_i^1, f_i^2)^\top, \quad \text{for every $v_i\in\V$}.
$$
Collecting each component of these into a vector in $\mathbb{R}^{n}$ gives
\begin{equation}
\f^1 \equiv
\begin{bmatrix}
f_1^1 \\
\vdots\\
f_n^1 
\end{bmatrix}, \qquad
\f^2 = 
\begin{bmatrix}
f_1^2 \\
\vdots\\
f_n^2
\end{bmatrix} 
\label{eq:f_matrix}
\end{equation}
For any point $v \in \tau = [{v}_{i}, {v}_{j}, {v}_{k}]$, the restriction of $f$ to $\tau$ is expressed in barycentric form
$$
f|_{\tau}(v) = \lambda_i(\tau,v) \, f_i + \lambda_j(\tau,v) \, f_j + \lambda_k(\tau,v) \, f_k, 
$$
where
$$
\lambda_i(\tau,v) = \frac{|[v, v_j, v_k]|}{|\tau|}, \quad
\lambda_j(\tau,v) = \frac{|[v_i, v, v_k]|}{|\tau|}, \quad
\lambda_k(\tau,v) = \frac{|[v_i, v_j, v]|}{|\tau|},
$$
and $|\tau|=\int_\tau \,\d A$ denotes the area of the triangle $\tau$ (see Figure~\ref{fig:barycentric}). 
Equivalently, the mapping can be written using the piecewise linear hat functions $\varphi_i:\M\to\mathbb{R}$,
$$
f(v) = \sum_{i=1}^n \varphi_i(v)\, f_i,
\quad
\text{where}
\quad
\varphi_i(v) = 
\begin{cases}
\lambda_i(\tau,v), & v \in \tau, \\
0, & \text{otherwise}.
\end{cases}
$$
Therefore, $f(v)$ can be reformulated in matrix form as 
\[
f(v) = 
\begin{bmatrix}
{\f^1}^\top \Phi(v)  &  {\f^2}^\top \Phi(v) 
\end{bmatrix},
\qquad
\Phi(v) = \begin{bmatrix}
\varphi_1(v) \\
\vdots \\
\varphi_n(v)
\end{bmatrix} \in \mathbb{R}^n.
\]

A simplicial map $f$ is said to be authalic or area-preserving up to a global scaling if
\begin{equation} \label{eq:area-preserving}
|f(\tau)| = c \, |\tau|, ~~ \mbox{for all}~ \tau \in \mathcal{F},
\end{equation}
where $c>0$ is a global constant. 
In fact, the constant $c = \A(f) / |\M|$, in which $|\M|$ and $\A(f)$ are the areas of the surface and image, respectively, given by
\begin{equation}
|\M| = \sum_{\tau \in \F} |\tau| \quad~ \mbox{and} \quad~ \A(f) = \sum_{\tau \in \F} |f(\tau)|.
\end{equation}

\begin{figure}
\centering
\begin{tikzpicture}[black,thick,scale=1.5]
\coordinate (v_i) at (0,   0);
\coordinate (v_j) at (2,   0);
\coordinate (v_k) at (1,1.73);
\coordinate (v)   at (0.8,0.5);
\filldraw[chocolate!55] (v_i) -- (v_j) -- (v);
\filldraw[cadetblue!65] (v_j) -- (v_k) -- (v);
\filldraw[darktaupe!65] (v_k) -- (v_i) -- (v);
\draw{
(v_i) -- (v_j) -- (v_k) -- (v_i)
};
\draw[dashed, black]{
(v) -- (v_i)
};
\draw[dashed, black]{
(v) -- (v_j)
};
\draw[dashed, black]{
(v) -- (v_k)
};
\tikzstyle{every node}=[circle, draw, fill=lightgoldenrodyellow!30,
                        inner sep=1pt, minimum width=2pt]                  
\draw[black]{
(0,   0) node{${v}_{i}$}
(2,   0) node{${v}_{j}$}
(1,1.73) node{${v}_{k}$}
};
\draw[black]{
(0.8,0.5) node{$~v~$}
};
\end{tikzpicture}
\caption{An illustration of the barycentric coordinates on a triangular face.}
\label{fig:barycentric}
\end{figure}
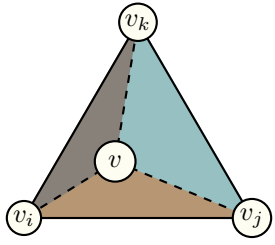

\section{Slicing method for square domains}
\label{sec:3}
A genus-zero or genus-one closed triangular mesh can be parameterized onto a square domain with appropriate boundary constraints after applying a suitable surface slicing. In this section, we propose a slicing method and specify the boundary conditions for the resulting square domain.

\subsection{Slicing closed surface}
To parameterize a closed surface onto a planar domain, we first introduce a cutting path on the surface. For a genus-zero surface, in order to minimize geometric distortion in the area-preserving parameterization, the path is chosen to follow the geodesic between the most distant pair of vertices. This pair can be approximated as the vertices with the maximum and minimum projections along the first principal component of all vertices.

Specifically, we compute the principal components matrix $R = [\r_1,~\r_2,~\r_3]$ via singular value decomposition for the vertex matrix,
\[
V = 
\begin{bmatrix}
   v_1^1  & v_1^2  &  v_1^3\\
   \vdots & \vdots & \vdots\\
   v_n^1  & v_n^2  &  v_n^3
\end{bmatrix} = U \Sigma R^\top.
\]
We then rotate the vertices by $\widetilde V = V R$ to align the principal components with the standard coordinate axes. The two most distant vertices along the first principal axis are identified as
\[
(v_m, v_M) \equiv \Big(\argmin_{1\leq i \leq n}  \widetilde v_{i}^1, \argmax_{1\leq i \leq n} \widetilde v_{i}^1 \Big),
\]
where $\widetilde v_{i}^j$ denotes the $(i,j)$th entry of $\widetilde V$. 
The geodesic path between $v_m$ and $v_M$ is then computed using the MATLAB function \texttt{shortestpath}, denoted as $\tp = \{m, \ldots, M \}$.

Next, given the cutting path $\tp$, we slice the triangular mesh along this path by duplicating the vertices on $\tp$ and updating the adjacent faces accordingly.
Formally, we first define the set of faces that contain exactly two vertices from the cutting path $\tp$ as
\begin{equation}
\F_{\tp}^2 = \left \{ \tau \in \F ~\big|~ v_i, v_j, v_k \in \tau,~~  i, j \in \tp, ~~ k \notin \mathtt{p} \right\}.
\label{eq:face_p2}
\end{equation}
We then classify these faces into two subsets based on their orientation relative to the tangent direction of the cutting path. Faces on the left side of the path are defined as
\begin{equation}
\F_\tp^\tL = \left \{ \tau \in \F_{\tp}^2 ~\big|~  \tau = [v_{\tp_i}, v_{\tp_{i+1}}, v_\ell] \right\},
\label{eq:face_p2_L}
\end{equation}
while faces on the right side are those with the opposite orientation:
\begin{equation}
\F_\tp^\tR = \left \{ \tau \in \F_\tp^2 ~\big|~  \tau = [v_{\tp_{i+1}}, v_{\tp_i}, v_\ell] \right\}.
\label{eq:face_p2_R}
\end{equation}

Moreover, we identify the set of faces that contain exactly one vertex on the cutting path:
\begin{equation}
\F_\tp^1 = \left \{  \tau \in \F ~\big|~ v_i, v_j, v_k \in \tau,~~ i \in \tp, ~~ j,k \notin \tp \right\}.
\label{eq:face_p1}
\end{equation}
Let $\V_{\tp}^\tL$ and $\V_{\tp}^\tR$ denote the sets of all vertices belonging to the left-side and right-side faces $\F_\tp^\tL$ and $\F_\tp^\tR$, respectively, excluding the vertices that lie on the path $\tp$. That is,
\begin{align}
\V_\tp^\tL &= \{ v_i \in \V ~\big|~ v_i \in \tau \in \F_\tp^\tL, ~ i\notin \tp \}, \label{eq:faceV_p1_L} \\
\V_\tp^\tR &= \{ v_i \in \V~\big|~ v_i \in \tau \in \F_\tp^\tR, ~ i\notin \tp \}. \label{eq:faceV_p1_R}
\end{align}
We then classify the faces in $\F_\tp^1$ according to their adjacency. A face is assigned to the left side if it contains a vertex belonging to $\V_\tp^\tL$, and to the right side if it contains a vertex belonging to $\V_\tp^\tR$. Formally,
\begin{align}
\F_\tp^\tL &\gets \F_\tp^\tL \cup \left \{ \tau \in \F_\tp^1 \setminus (\F_{\tp}^\tL \cup \F_\tp^\tR) ~\big|~ \tau \cap \V_\tp^\tL \neq \varnothing \right\},  \label{eq:face_p1_L}\\
\F_\tp^\tR &\gets \F_\tp^\tR\cup \left \{ \tau \in \F_\tp^1 \setminus (\F_\tp^\tL \cup \F_\tp^\tR) ~\big|~  \tau \cap \V_\tp^\tR \neq \varnothing\right\}. \label{eq:face_p1_R}
\end{align}
This process is repeated until all faces in $\F_\tp^1$ have been assigned, i.e.,
$$ \F_\tp^1 \setminus (\F_\tp^\tL \cup \F_\tp^\tR) = \varnothing.$$ 
It is important to alternate between the two steps \eqref{eq:face_p1_L} and \eqref{eq:face_p1_R} to ensure that all faces in $\F_\tp^1$ are correctly assigned to either $\F_\tp^\tL$ or $\F_\tp^\tR$.
Once the classification is complete, we duplicate the path vertices $v_\tp$ to create a corresponding set $\hat{v}_\tp$. For each face $\tau = [v_i, v_j, v_k] \in \F_\tp^\tR$ where $i \in \tp$, we replace $v_i$ with its duplicate $\hat{v}_i$. The complete slicing procedure is summarized in Algorithm~\ref{alg:slice_mesh}.

For genus-one closed surfaces, the cutting path is given by the union of the handle and tunnel loops, which can be extracted using the Reeb graph–based method of Dey et al.~\cite{DeLi08}. The mesh is then sliced open into a disk-like mesh by applying the slicing procedure described above to each of these loops.

As a result, we obtain a disk-like surface by cutting the original closed surface. For simplicity, we continue to denote the resulting triangular mesh as $\T = (\V, \E, \F)$, as defined in \eqref{eq:mesh}, now representing the simplicial open surface.

\begin{algorithm}[]
\caption{Slicing a triangular mesh along a path}
\label{alg:slice_mesh}
\begin{algorithmic}[1]
\Require A triangular mesh $\M = (\F, \V)$, cutting path $\mathtt{p}$
\Ensure Sliced triangular mesh $\hat{\M} = (\F, \hat \V)$
\State Compute adjacent face set $\F_\tp^2$ as in \eqref{eq:face_p2}
\State Split $\F_\tp^2$ into left and right subsets: $\F_\tp^\tL$ and $\F_\tp^\tR$ via \eqref{eq:face_p2_L}, \eqref{eq:face_p2_R}
\State Compute adjacent face set $\F_\tp^1$ as in \eqref{eq:face_p1}
\While{$\F_\tp^1 \setminus (\F_\tp^\tL \cup \F_\tp^\tR) \neq \varnothing$}
    \State Identify vertex sets $\V_\tp^\tL$, $\V_\tp^\tR$ via \eqref{eq:faceV_p1_L}, \eqref{eq:faceV_p1_R}
    \State Update $\F_\tp^\tL$ and $\F_\tp^\tR$ using \eqref{eq:face_p1_L}, \eqref{eq:face_p1_R}
\EndWhile
\State Duplicate each path vertex $v_i \in \mathtt{p}$ as $\hat{v}_i$
\State Replace $v_i \in \tp$ with $\hat{v}_i$ for all faces in $\F_\tp^\tR$
\end{algorithmic}
\end{algorithm}

\subsection{Boundary conditions of square domains}
After slicing closed genus-zero and genus-one surfaces, the resulting open surface is parameterized onto a unit square domain subject to appropriate boundary constraints, as illustrated in Figure~\ref{fig:plan_close_mesh}.

Let $\B = \{b \mid v_b \in \partial \M\}$ denote the set of boundary vertex indices. Among these, we identify four corner indices, denoted by $\B_i$, $\B_j$, $\B_k$, and $\B_\ell$, which are mapped to the corners of the unit square:
\begin{align*}
\f_{\B_i} &= (0,0) ,~~~ \f_{\B_j} = (1,0),\\
\f_{\B_k} &= (1,1) ,~~~ \f_{\B_\ell} = (0,1).
\end{align*}
Since the boundary indices are ordered cyclically, we assume without loss of generality that $i = 1$.
We partition the boundary vertices into four ordered sets corresponding to the square edges:
\begin{subequations} \label{eq:bdry_edge}
\begin{align}
\tE &= \{\B_{i}, \B_{i+1}, \ldots, \B_{j}\},\\
\tF &= \{\B_{j}, \B_{j+1}, \ldots, \B_{k}\},\\
\tG &= \{\B_{\ell}, \B_{\ell-1}, \ldots, \B_{k}\},\\
\tH &= \{\B_{i}, \B_{n_\B}, \B_{n_\B-1},\ldots, \B_{\ell}\},
\end{align}
\end{subequations}
where $n_\B$ is the number of boundary vertices.

We then impose boundary constraints to align these boundary segments with the square edges. For genus-zero surfaces, the conditions are:
\begin{subequations} \label{eq:bdry_cond}
\begin{align}
\f_{\tE}^2 &= \0, ~\quad \f_{\tH}^1 = \0, \quad \f_{\tG}^2 = \1,~~\f_{\tF}^1 = \1,\\
\f_{\tE}^1 &= \f_{\tH}^2 ,\quad \f_{\tG}^1  = \f_{\tF}^2,
\end{align}
while for genus-one surfaces, the constraints become:
\begin{align}
\f_{\tE}^2 &= \0, ~\quad \f_{\tH}^1 = \0, \quad \f_{\tG}^2 = \1, \quad \f_{\tF}^1 = \1,\\
\f_{\tE}^1 &= \f_{\tG}^1, \quad \f_{\tF}^2 = \f_{\tH}^2.
\end{align}
\end{subequations}

With the slicing complete and boundary constraints defined, we are able to compute an area-preserving parameterization by minimizing the authalic energy, as described in the following section.

\begin{figure}[]
\centering
\begin{tabular}{cc}
\begin{tikzpicture}
\tikzset{arrowMe/.style={postaction=decorate, decoration={markings, mark=at position .6 with {\arrow[very thick]{#1}}} }}
\def\r{1.7}
\tkzDefPoint(0,0){o}
\tkzDefPoint(45:\r){v1}\tkzDefPoint(135:\r){v2}\tkzDefPoint(225:\r){v3}\tkzDefPoint(315:\r){v4}
\draw[arrowMe=stealth](v2)--node[above=2pt]{\(b\)}(v1);
\draw[arrowMe=stealth](v4)--node[right]{\(b\)}(v1);
\draw[arrowMe=stealth](v3)--node[left]{\(a\)}(v2);
\draw[arrowMe=stealth](v3)--node[below]{\(a\)}(v4);
\tkzDrawPoints(v1,v2,v3,v4)
\end{tikzpicture} & 
\begin{tikzpicture}
\tikzset{arrowMe/.style={postaction=decorate, decoration={markings, mark=at position .6 with {\arrow[very thick]{#1}}} }}
\def\r{1.7}
\tkzDefPoint(0,0){o}
\tkzDefPoint(45:\r){v1}\tkzDefPoint(135:\r){v2}\tkzDefPoint(225:\r){v3}\tkzDefPoint(315:\r){v4}
\draw[arrowMe=stealth](v2)--node[above=2pt]{\(a\)}(v1);
\draw[arrowMe=stealth](v4)--node[right]{\(b\)}(v1);
\draw[arrowMe=stealth](v3)--node[left]{\(b\)}(v2);
\draw[arrowMe=stealth](v3)--node[below]{\(a\)}(v4);
\tkzDrawPoints(v1,v2,v3,v4)
\end{tikzpicture} \\
genus-zero surface & genus-one surface
\end{tabular}
\caption{Boundary constraints for square domain parameterizations of sliced genus-zero and genus-one surfaces.}
\label{fig:plan_close_mesh}
\end{figure}
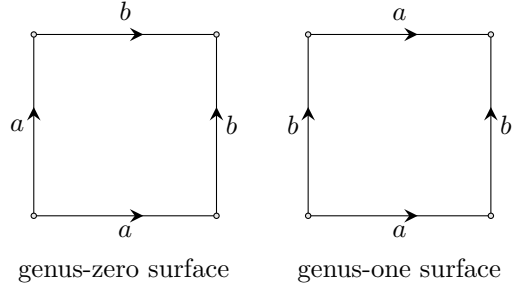

\section{Authalic energy minimization}
\label{sec:4}
Authalic energy is a functional that quantifies the area distortion of a simplicial mapping. Minimizing this energy yields an area-preserving parameterization. In this section, we present its properties and the corresponding gradient formula.

\subsection{Authalic energy functional}
To quantify area distortion, Yueh \cite{Yueh23} introduces the \textit{stretch energy functional} for a simplicial mapping $f$, defined as
\[
E_\rS(f) = \sum_{\tau \in \F}\frac{|f(\tau)|^2}{|\tau|}.
\]
It was shown that $E_\rS(f)$ is greater than or equal to the image area $\A(f)$, with equality achieved when $f$ is area-preserving, under the assumption that the total area is preserved, i.e., $\A(f) = |\M|$. 
In practice, however, the total area can vary when boundary conditions are not fixed. To address this limitation, Liu and Yueh \cite{LiYu24} proposed the authalic energy functional, defined as
\begin{equation}
E_\rA(f) = \frac{|\M|}{\A(f)} E_\rS(f) - \A(f),
\label{eq:Ea}
\end{equation}
which relaxes the assumption of the image area, as shown in the following theorem.

\begin{theorem}[{{\cite[Theorem 1]{LiYu24}}}]
\label{thm:Ea_Cauchy}
Let $f$ be a simplicial mapping defined on $\M$. The authalc energy functional $E_\rA$ in \eqref{eq:Ea} satisfies $E_\rA \geq 0$, and equality holds if and only if $f$ is area-preserving.
\end{theorem}
\begin{proof}
By the Cauchy--Schwarz inequality, we have
\[
E_\rS(f) \, |\M| = \sum_{\tau \in \F}\frac{|f(\tau)|^2}{|\tau|} \sum_{\tau \in \F}|\tau| \geq \Big(\sum_{\tau \in \F} |f(\tau)| \Big)^2 = \A(f)^2.
\]
Dividing by $\A(f)$ and rearranging yields $E_\rA(f) \geq 0$. Equality occurs when 
$\frac{|f(\tau)|}{|\tau|} = c$ is constant for all $\tau \in \F$, 
which characterizes an area-preserving map.
\end{proof}

Although Theorem~\ref{thm:Ea_Cauchy} ensures that minimizing $E_\rA$ enforces area preservation, achieving $E_\rA = 0$ exactly is practically unattainable. This naturally raises the question of whether sufficiently reducing $E_\rA$ guarantees acceptable preservation. In fact, $E_\rA$ characterizes the area-weighted variance of per-triangle area ratio, as established in the following theorem.

\begin{theorem} \label{thm:Ea_var}
Let the area ratio on each triangular face $\tau \in \F$ be defined as
\begin{equation}
r_f(\tau) = \frac{|f(\tau)|}{|\tau|}
\label{eq:area_ratio}
\end{equation}
Then, the mean and variance of $r_f(\tau)$ with respect to the area weights $\tfrac{|\tau|}{|\M|}$ are given by
\begin{align}
\widetilde \mu (r_f) &\equiv \sum_{\tau\in\mathcal{F}} \frac{|\tau|}{|\M|} r_f(\tau)  =  \frac{\A(f)}{|\M|}, \label{eq:area_weight_mean} \\
\widetilde  s^2 (r_f) &\equiv  \sum_{\tau \in \F} \frac{|\tau|}{|\M|} \big( r_f(\tau) - \widetilde \mu(r_f) \big)^2 
= \frac{\A(f)}{|\M|^2} E_\rA(f),
\label{eq:area_weight_var}
\end{align}
\end{theorem}

\begin{proof}
By definition, the stretch energy $E_\rS(f)$ and the image area $\A(f)$ can be expressed in terms of the area ratios $r_f(\tau)$ as
\begin{align*}
E_\rS(f) &= \sum_{\tau \in \F} \frac{|f(\tau)|^2}{|\tau|} = \sum_{\tau \in \F} \bigg( \frac{|f(\tau)|}{|\tau|} \bigg)^2 |\tau| \\
&= \sum_{\tau \in \F} r_f(\tau)^2  |\tau| = |\M| \bigg( \sum_{\tau \in \F} \frac{|\tau|}{|\M|}r_f(\tau)^2 \bigg) = |\M|\, \widetilde \mu (r_f^2), \\
\A(f)&= \sum_{\tau \in \F} |f(\tau)| = \sum_{\tau \in \F} \bigg(  \frac{|f(\tau)|}{|\tau|} \bigg) |\tau| = \sum_{\tau \in \F} r_f(\tau)  |\tau| \\
&= |\M| \bigg( \sum_{\tau \in \F} \frac{|\tau|}{|\M|} r_f(\tau) \bigg)  = |\M|\,\widetilde  \mu (r_f). 
\end{align*}
It follows that the variance of the area ratios with area weights is
\begin{align*}
\widetilde s^2 (r_f) &= \sum_{\tau \in \F} \frac{|\tau|}{|\M|} \big( r_f(\tau) - \widetilde \mu(r_f) \big)^2 
= \widetilde \mu (r_f^2) - \widetilde \mu(r_f)^2 \\
&= \frac{E_\rS(f)}{|\M|} - \frac{\A(f)^2}{|\M|^2} 
= \frac{\A(f)}{|\M|^2} \Big( \frac{|\M|}{\A(f)}E_\rS(f) - \A(f) \Big) = \frac{\A(f)}{|\M|^2} E_\rA(f).
\end{align*}
\end{proof}

In numerical experiments, area distortion is typically measured by the standard deviation of the area ratios, i.e., the square root of their variance. The following result shows that this variance is also bounded above in terms of $E_\rA$.

\begin{theorem} \label{thm:Ea_unwvar}    
Let $r_f(\tau)$ denote the area ratio of a triangular face $\tau \in \F$, as defined in \eqref{eq:area_ratio}, and let $s^2(r_f)$ denote the (unweighted) variance of $ r_f(\tau) $. Then
\[
s^2 (r_f) \leq  \frac{\A(f)}{m\,|\M| \displaystyle \min_{\tau \in \F}|\tau|} E_\rA(f),
\]
where $m = |\F|$ is the number of faces of $\M$.
\end{theorem}
\begin{proof}
The (unweighted) mean of the area ratios is defined as
$$\mu (r_f) \equiv \frac{1}{m} \sum_{\tau \in \F} r_f(\tau).$$
From the fact that $\mu(r_f) = \argmin_x \sum_{\tau \in \F} (r_f(\tau) - x)^2$, we have
\begin{align*}
s^2 (r_f)  
&= \frac{1}{m}\sum_{\tau \in \F} \big( r_f(\tau) - \mu(r_f) \big)^2 
\leq \frac{1}{m }\sum_{\tau \in \F} \big( r_f(\tau) - \widetilde \mu (r_f) \big)^2 \\
&\leq \frac{|\M|}{m \displaystyle \min_{\tau \in \F}|\tau|}\sum_{\tau \in \F}  \frac{|\tau|}{|\M|} \big( r_f(\tau) -  \widetilde \mu (r_f) \big)^2
=  \frac{|\M|}{m \displaystyle \min_{\tau \in \F}|\tau|} \widetilde s^2 (r_f),
\end{align*}
where $\widetilde \mu(r_f)$ and $\widetilde s^2(r_f)$ are the area-weighted mean and variance, as defined in \eqref{eq:area_weight_mean} and \eqref{eq:area_weight_var}, respectively.
As a result, by Theorem~\ref{thm:Ea_var}, we obtain the claimed inequality:
\[
s^2 (r_f) \leq \frac{|\M|}{m \displaystyle \min_{\tau \in \F}|\tau|} \widetilde s^2 (r_f) = \frac{\A(f)}{m \, |\M|\displaystyle \min_{\tau \in \F}|\tau|} E_\rA(f).
\]
\end{proof}

These results not only indicate that it is reasonable to characterize the variance of area ratios using the authalic energy, but also explain that the variance tends to decrease as the authalic energy diminishes, as confirmed numerically in Section~\ref{sec:6.2}.

\begin{figure}[]
\centering
\begin{tikzpicture}[thick,scale=1.5]
\coordinate (v_i) at (0,0);
\coordinate (v_j) at (0,2);
\coordinate (v_k) at (2,1);
\coordinate (v_l) at (-2,1);
\filldraw[camel!60] (v_i) -- (v_j) -- (v_k);
\filldraw[camel!60] (v_i) -- (v_j) -- (v_l);
\pic[draw, ->, "$\theta_{i,j}^k(f)$", angle eccentricity=2.3, angle radius=0.5cm]{angle = v_i--v_l--v_j};
\pic[draw, ->, "$\theta_{j,i}^\ell(f)$", angle eccentricity=2.3, angle radius=0.5cm]{angle = v_j--v_k--v_i};
\draw{
(v_i) -- (v_j) -- (v_k) -- (v_i) -- (v_l) -- (v_j)
};
\tikzstyle{every node}=[circle, draw, fill=lightgoldenrodyellow!30,
                        inner sep=1pt, minimum width=2pt]
\draw{
(0,0) node{$f_i$}
(0,2) node{$f_j$}
(2,1) node{$f_\ell$}
(-2,1) node{$f_k$}
};
\end{tikzpicture}
\caption{An illustration of the angle $\theta_{i,j}^k(f)$ and $\theta_{i,j}^\ell(f)$ defined on the image $f(\M)$.}
\label{fig:cot}
\end{figure}
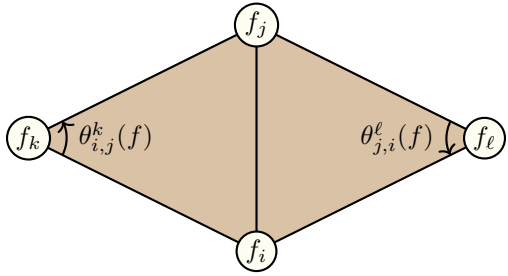

\subsection{Gradient under boundary conditions}
For the square-domain parameterization, the image area $\A(f)$ is equal to $1$. Therefore, minimizing energy $E_\rA(f)$ is equivalent to minimizing energy $E_\rS(f)$.
Using the matrix representation of the simplicial map $f$ in \eqref{eq:f_matrix}, the energy $E_\rS$ is reformulated as 
\begin{equation}
E_\rS(f)  = \frac{1}{2} \sum_{s = 1}^2 {\f^s}^\top L_\rS(f) \, \f^s
\label{eq:Es}
\end{equation}
where $L_\rS(f)$ is the weighted Laplacian matrix, defined by
\begin{equation} \label{eq:Ls}
L_\rS(f)_{i, j} = 
\begin{cases}
    \displaystyle -\sum_{\tau = [v_i,v_j,v_k]} \frac{\cot(\theta_{i, j}^k(f))}{2} \frac{ |f(\tau)|}{|\tau|}   &\text{if $[v_i, v_j] \in \mathcal{E}$,} \\[0.5em]
    \displaystyle-\sum_{\ell \neq i} L_\rS(f)_{i, \ell}     &\text{if $j=i$,} \\
    0       &\text{otherwise},
\end{cases}
\end{equation}
with $\theta_{i, j}^k(f)$ denoting the angle opposite edge $[\f_i, \f_j]$ at vertex $\f_k$ (see Figure~\ref{fig:cot}), as shown in \cite[Lemma 3.1]{Yueh23}.
Furthermore, the gradient of $E_\rS$ with respect to $\f^s$ is given by
\begin{equation}
\nabla_{\f^s} E_\rS(f) = 2\, L_\rS(f) \,\f^s,~~~s = 1,2,
\label{eq:dEs}
\end{equation}
as stated in \cite[Theorem 3.5]{Yueh23}.

Let the index set of interior vertices be $\I = \{1, \ldots, n\} \setminus \B$. By reordering the vectors $\f^s$ and the matrix $L_\rS(f)$ accordingly, we can write
\[
L_\rS(f) = 
\begin{bmatrix}
    L_\rS(f)_{\I,\I}  &  L_\rS(f)_{\I,\B}\\
    L_\rS(f)_{\B,\I}  &  L_\rS(f)_{\B,\B}
\end{bmatrix}, \qquad
\f^s = 
\begin{bmatrix}
\f_\I^s \\
\f_\B^s
\end{bmatrix},
\]
for $s = 1,2$. By \eqref{eq:dEs}, the gradient of $E_\rS(f)$ with respect to the interior and boundary mappings is then given by
\begin{subequations}\label{eq:grad_Es}
\begin{align}
\nabla_{\f_\I^s} E_\rS &= 2\, \big( L_\rS(f)_{\I,\I} \, \f_\I^s + L_\rS(f)_{\I,\B} \, \f_\B^s \big), \label{eq:dEs_VI} \\
\nabla_{\f_\B^s} E_\rS &= 2\, \big( L_\rS(f)_{\B,\I} \, \f_\I^s + L_\rS(f)_{\B,\B} \, \f_\B^s \big), \label{eq:dEs_VB}
\end{align}
for $s = 1,2$.

To enforce the boundary constraints in \eqref{eq:bdry_cond}, we treat $\f_{\tE}^1$ and $\f_{\tF}^2$ as variables, while updating $\f_{\tG}^1$ and $\f_{\tH}^2$ according to the boundary constraints. Define the complements $\tE^c = \{1, \ldots, n\} \setminus \tE$, and similarly for $\tF^c$, $\tG^c$, and $\tH^c$.
For genus-zero surfaces, the gradients of $E_\rS$ with respect to $\f_{\tE}^1$ and $\f_{\tF}^2$ are given by:
\begin{align}
\nabla_{\f_{\tE}^1} E_\rS(f) &= 2\, \big( L_\rS(f)_{\tE, \tE^c} \, \f_{\tE^c}^1 + L_\rS(f)_{\tE,\tE} \, \f_{\tE}^1 \big) \nonumber\\
&+ 2\, \big( L_\rS(f)_{\tH, \tH^c} \, \f_{\tH^c}^2 + L_\rS(f)_{\tH,\tH} \, \f_{\tH}^2 \big),\\
\nabla_{\f_{\tF}^2} E_\rS(f) &= 2\, \big( L_\rS(f)_{\tF, \tF^c} \, \f_{\tF^c}^2 + L_\rS(f)_{\tF,\tF} \, \f_{\tF}^2 \big) \nonumber\\
&+ 2\, \big( L_\rS(f)_{\tG, \tG^c} \, \f_{\tG^c}^1 + L_\rS(f)_{\tG,\tG} \, \f_{\tG}^1 \big).
\end{align}
For genus-one surfaces, the gradients become:
\begin{align}
\nabla_{\f_{\tE}^1} E_\rS(f) &= 2\, \big( L_\rS(f)_{\tE, \tE^c} \, \f_{\tE^c}^1 + L_\rS(f)_{\tE,\tE} \, \f_{\tE}^1 \big) \nonumber\\
&+ 2\, \big( L_\rS(f)_{\tG, \tG^c} \, \f_{\tG^c}^1 + L_\rS(f)_{\tG,\tG} \, \f_{\tG}^1 \big),\\
\nabla_{\f_{\tF}^2} E_\rS(f) &= 2\, \big( L_\rS(f)_{\tF, \tF^c} \, \f_{\tF^c}^2 + L_\rS(f)_{\tF,\tF} \, \f_{\tF}^2 \big)\nonumber\\
&+ 2\, \big( L_\rS(f)_{\tH, \tH^c} \, \f_{\tH^c}^2 + L_\rS(f)_{\tH,\tH} \, \f_{\tH}^2 \big).
\end{align}
\end{subequations}
With these gradient formulations, we can apply an optimization method to minimize $E_\rS$.

\section{Energy minimization methods}
\label{sec:5}
In Section~\ref{sec:4}, we established the area-preservation properties of the authalic energy $E_\rA$. For the square parameterization, we may normalize the surface area $|\M|$ so that $|\M| = \A(f) = 1$, and minimizing $E_\rA$ is equivalent to minimizing the stretch energy $E_\rS(f)$. In this section, we present an efficient algorithm for minimizing $E_\rS(f)$ subject to the boundary constraint \eqref{eq:bdry_cond}.

\subsection{Fixed-point method for initial map}
In the work that first proposed the stretch energy \cite{YuLW19}, they minimize the energy by the fixed-point method, which is based on approximating the critical point with respect to the interior map $\f_\I^s$ for $s = 1,2$. In particular, with a fixed boundary map $\f_\B^s$, the fixed-point method iteratively solves the linear system:
\begin{equation}
L_\rS(f^{(k)})_{\I, \I} {\f_\I^s}^{(k+1)} = \,L_\rS(f^{(k)})_{\I, \B}\, \f_\B^s, ~~ s = 1,2, 
\label{eq:FPM_Es}
\end{equation}
until convergence.

As expected, the fixed-point method updates the map by minimizing a sequence of quadratic functions that locally approximate $E_\rS$. This yields second-order accuracy with respect to energy variation, as stated below.

\begin{theorem} \label{thm:FPM_quad}
Let $f$ be a simplicial map, and let ${\f^{(k)}}$ denote the sequence generated by the fixed-point method \eqref{eq:FPM_Es}. Then, 
\[
\left| E_\rS(f^{(k+1)}) - E_\rS(f^{(k)}) \right| = \O \left( \|\f^{(k+1)} - {\f}^{(k)} \|^2_F \right).
\]
\end{theorem}
\begin{proof}
We denote $h = \|\f - {\f}^{(k)} \|_F$. By the Taylor expansion of $E_\rS(f)$ at $\f^{(k)}$ and \eqref{eq:dEs}, we have 
\begin{align}
E_\rS(f) &= \frac{1}{2} \sum_{s = 1}^{2} {\f^s}^\top L_\rS(f) \, \f^s  \nonumber\\
&=  \frac{1}{2} \sum_{s = 1}^{2} {{\f^s}^{(k)}}^\top L_\rS(f^{(k)}) \, {\f^s}^{(k)}  + 2\sum_{s = 1}^{2}(\f^s - {\f^s}^{(k)})^\top L_\rS(f^{(k)}) \, {\f^s}^{(k)} + \O ( h^2 ). \label{eq:Taylor_Es}
\end{align}   
The linear term can be rewritten as 
\begin{align*}
(\f^s - {\f^s}^{(k)})^\top &L_\rS(f^{(k)}) \, {\f^s}^{(k)}  =   {\f^s}^\top L_\rS(f^{(k)}) \, {\f^s}^{(k)} - {{\f^s}^{(k)}}^\top L_\rS(f^{(k)}) \, {\f^s}^{(k)} \\
&= {\f^s}^\top L_\rS(f^{(k)}) \, \f^s + {\f^s}^\top L_\rS(f^{(k)}) ({\f^s}^{(k)} - \f^s)- {{\f^s}^{(k)}}^\top L_\rS(f^{(k)}) \, {\f^s}^{(k)}.
\end{align*}
We define the quadratic functional as
\begin{equation}
\Q^{(k)}(f) =  2 \sum_{s = 1}^{2} {\f^s}^\top L_\rS(f^{(k)}) \, \f^s -\frac{3}{2} \sum_{s = 1}^{2} {{\f^s}^{(k)}}^\top L_\rS(f^{(k)}) \, {\f^s}^{(k)}.
\label{eq:quad_fun}
\end{equation}
Then, substituting \eqref{eq:quad_fun} into equation~\eqref{eq:Taylor_Es}, we obtain
\begin{equation} \label{eq:5.4}
E_\rS(f) - \Q^{(k)}(f) = 2\sum_{s = 1}^{2} {\f^s}^\top L_\rS(f^{(k)}) ({\f^s}^{(k)} - \f^s ) + \O(h^2).
\end{equation}

The interior mapping $\f_\I^{(k+1)}$ generated by the fixed-point method \eqref{eq:FPM_Es} is a critical point of $\Q^{(k)}(f)$ with respect to $\f_\I$ and satisfies $(L_\rS(f^{(k)}) {\f^s}^{(k+1)})_\I = \0$. We let the direction vector $\bd^{(k)} = {\f^s}^{(k+1)} - {\f^s}^{(k)}$. Since $\f_\B$ is fixed, $\bd^{(k)}_\B=\0$. It follows that
\begin{align}
({\f^s}^{(k)} - {\f^s}^{(k+1)})^\top &L_\rS(f^{(k)})  \,{\f^s}^{(k+1)}  
= - {\bd^{(k)}}^\top  L_\rS(f^{(k)})  \,{\f^s}^{(k+1)} \nonumber \\
&=-{\bd^{(k)}_\I}^\top  \big( L_\rS(f^{(k)}) \,{\f^s}^{(k+1)} \big)_\I  -  {\bd^{(k)}_\B}^\top \big( L_\rS(f^{(k)}) \,{\f^s}^{(k+1)} \big)_\B 
= \0. \label{eq:5.4.1}
\end{align}
By substituting \eqref{eq:5.4.1} into \eqref{eq:5.4}, we obtain
\begin{equation}
E_\rS(f^{(k+1)}) - \Q^{(k)}(f^{(k+1)}) = \O (h^2).  \label{eq:Es_and_Q}
\end{equation}
Consequently, by using \eqref{eq:Es_and_Q} and the fact that $\Q^{(k)}(f^{(k)}) = E_\rS(f^{(k)})$, we obtain
\begin{align*}
E_\rS(f^{(k)}) - E_\rS(f^{(k+1)}) 
&= \left( E_\rS(f^{(k)}) -  \Q^{(k)}(f^{(k+1)}) \right) + \left(  \Q^{(k)}(f^{(k+1)}) - E_\rS(f^{(k+1)})  \right) \\
&= \underbrace{\left( \Q^{(k)}(f^{(k)}) -  \Q^{(k)}(f^{(k+1)}) \right)}_{\O(h^2)} + \O(h^2) 
= \O(h^2),
\end{align*}
which concludes the desired result.
\end{proof}

The quadratic approximation property of the fixed-point method is illustrated in Figure~\ref{fig:Es_and_Q}, where, for visualization, we set the interior set $\I$ to consist of only a single vertex. In each iteration, the fixed-point method approximates $E_\rS$ by a quadratic function $\Q^{(k)}$ in \eqref{eq:quad_fun} that matches both the function value and gradient direction at $\f_\I^{(k)}$.

It partially explains the numerical behavior that the energy decreases significantly during the initial iterations, but the improvement diminishes rapidly thereafter. It also confirms the numerical observation that the gradient descent method for $E_\rS$, which follows the local steepest descent direction of $\Q^{(k)}$, is less effective than the fixed-point method, which follows the exact minimizer of $\Q^{(k)}$ in each step. The complete algorithmic procedure is summarized in Algorithm~\ref{alg:FPM}.

While the fixed-point method is effective, it lacks a guarantee of global convergence. To remedy this drawback, we apply the fixed-point method for $10$ iterations, followed by a preconditioned nonlinear conjugate gradient (CG) method to continue the minimization process.

\begin{figure}[]
\centering
\includegraphics[height=5cm]{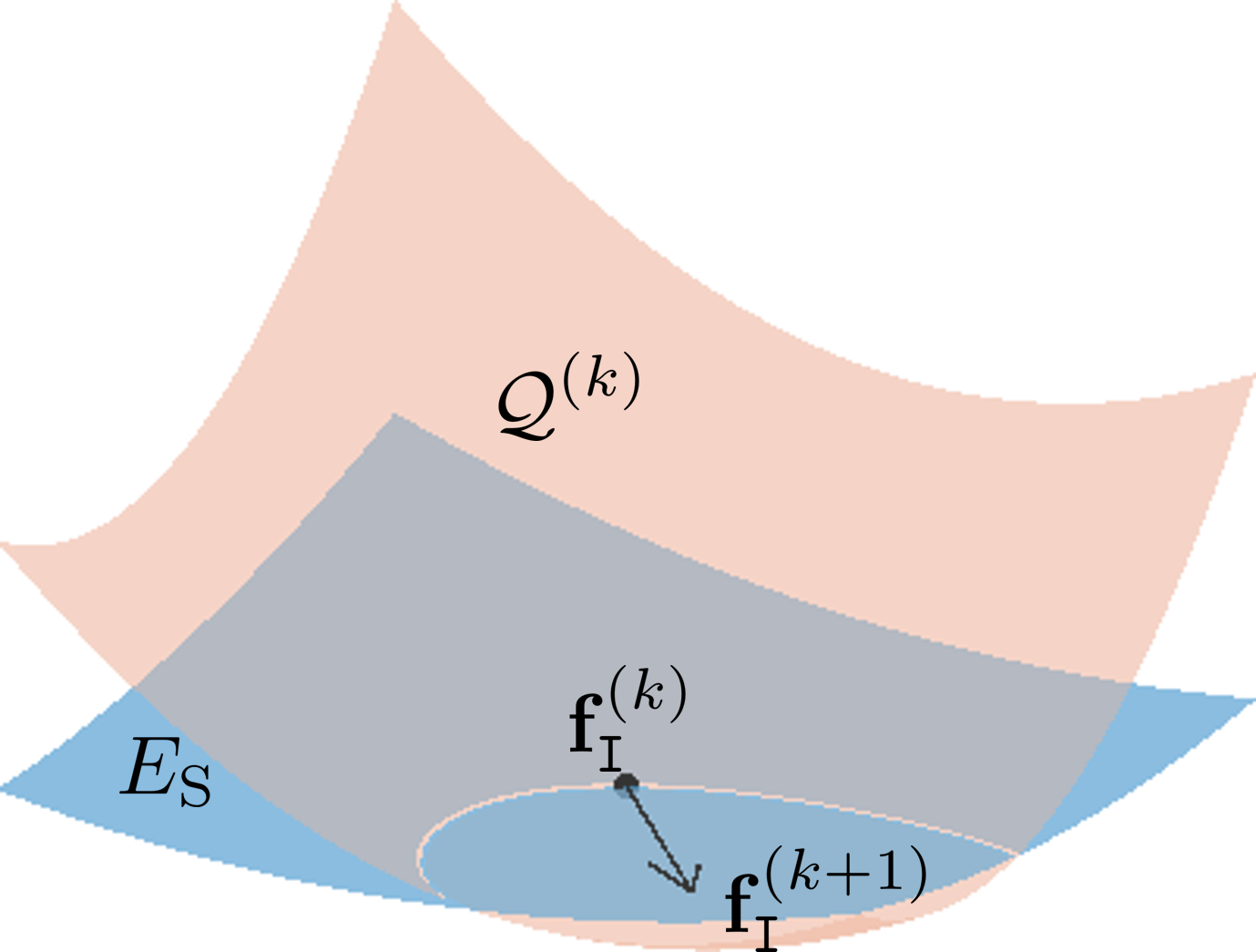}
\caption{Sequential quadratic approximations in the fixed-point method \eqref{eq:FPM_Es}. } 
\label{fig:Es_and_Q}
\end{figure}

\begin{algorithm}[h]
\caption{Fixed-point method for square initial mapping}
\label{alg:FPM}
\begin{algorithmic}[1]
\Require A simply connected open triangular mesh $\M$.
\Ensure A square area-preserving initial map $f$.
\State Let $\B = \{b \mid v_b \in \partial \M\}$ and $\I = \{1, \ldots, n\} \setminus \B$.
\State Let $\tE$, $\tF$, $\tG$, and $\tH$ be boundary segments as in \eqref{eq:bdry_edge}.
\State Let $\f_{\tE}^1 = (0,t,2t,\ldots,1)$ and $\f_{\tF}^2 = (0,s,2s,\ldots,1)$, for some $t, s \in \mathbb{R}$.
\State Compute the full boundary map $\f_\B$ using \eqref{eq:bdry_cond}, depending on the genus.
\State Initialize $L = L_\rS(\mathrm{id.})$ as in \eqref{eq:Ls}, with identity map.
\For{$k = 0$ to $10$}
    \State Solve the linear system $L_{\I,\I} \f^s_{\I} = -L_{\I,\B} \f^s_{\B}$, for $s=1,2$.
    \State Update $L$ using \eqref{eq:Ls}.
\EndFor
\end{algorithmic}
\end{algorithm}

\subsection{Preconditioned nonlinear CG method}
A common approach for unconstrained optimization is the \textit{line search method} \cite[Chapter 3]{NoWr06}, which iteratively selects a descent direction and determines an appropriate step size along this direction to achieve maximal reduction in the objective function. Specifically, we define the variable vector as
\[
\f = 
\begin{bmatrix}
    \mathrm{vec}(\f_\I) \\
    \f_{\mathtt{E}}^1\\
    \f_{\mathtt{F}}^2
\end{bmatrix} \in \mathbb{R}^{2n}
\]
and let $\g$ denote the gradient of $E_\rS(f)$ with respect to $\f$, given by \eqref{eq:grad_Es}. At the $k$th iteration, the update is along a direction vector $\p^{(k)}$ with a step size $\alpha_k > 0$, according to
\begin{equation}
\f^{(k+1)} = \f^{(k)} + \alpha_k \p^{(k)}.
\label{eq:line_search}
\end{equation}

A basic line search method is the steepest descent, which chooses the negative gradient as the search direction, $\p^{(k)} = -\g^{(k)}$. While this method guarantees global convergence under an appropriately chosen step size, it often exhibits slow convergence in practice.

An alternative is the CG method, originally developed for solving positive-definite linear systems and particularly effective for large, sparse systems \cite{HeSt52}. The method was later generalized to nonlinear optimization by Fletcher and Reeves \cite{FlRe64}. It improves upon steepest descent by introducing a correction term to the direction, given by
\[
\p^{(k)} = - \g^{(k)} + \frac{{\g^{(k)}}^\top \g^{(k)}} { {\g^{(k-1)}}^\top \g^{(k-1)}} \p^{(k-1)} 
\]
with the initialization $\p^{(0)} = -\g^{(0)}$. When the objective function is convex quadratic, the correction ensures that all the search directions are conjugate.

The fixed-point method in \eqref{eq:FPM_Es} can be reformulated in a line-search form as
\begin{align*}
{\f^s_\I}^{(k+1)} &= -  L_\rS(f^{(k)})_{\I, \I}^{-1} \,L_\rS(f^{(k)})_{\I, \B}\, {\f^s_\B} \\
    &= {\f^s_\I}^{(k)}-  L_\rS(f^{(k)})_{\I,\I}^{-1} ~ \big( L_\rS(f^{(k)})_{\I,\I} \, {\f^s_\I}^{(k)} + L_\rS(f^{(k)})_{\I,\B} \, \f^s_\B \big)\\
    &= {\f^s_\I}^{(k)} -  \tfrac{1}{2}\, L_\rS(f^{(k)})_{\I,\I}^{-1} \, {\g^s_\I}^{(k)},
\end{align*}
where ${\g^s_\I}^{(k)}$ is the gradient of $E_\rS$ with respect to ${\f^s_\I}^{(k)}$. 
This update is equivalent to the line-search method \eqref{eq:line_search} with a fixed step size of $\frac{1}{2}$, and with search direction $-L_\rS(f^{(k)})_{\I,\I}^{-1} \, {\g_\I^s}^{(k)}$, which corresponds to the negative gradient preconditioned by $L_\rS(f^{(k)})_{\I,\I}^{-1}$.
This observation motivates the use of a preconditioned nonlinear conjugate gradient method, aiming to take the advantages of both the fixed-point method and the CG method, where the direction is given by
\[
\p^{(k)} = - M^{-1}\g^{(k)} +  \frac{{\g^{(k)}}^\top M^{-1} \g^{(k)}} { {\g^{(k-1)}}^\top M^{-1} \g^{(k-1)}} \p^{(k-1)},
\]
with the initialization $\p^{(0)} = - M^{-1}\g^{(0)}$, and $M$ is the block diagonal preconditioner matrix given by
\begin{equation} \label{eq:precond}
M = 
\begin{bmatrix}
      I_2 \otimes L_\rS(f^{(0)})_{\I,\I} &  & \\
 &    L_\rS(f^{(0)})_{\tE,\tE} & \\
 & &  L_\rS(f^{(0)})_{\tF,\tF}    
\end{bmatrix}.
\end{equation}

It is worth noting that the preconditioner matrix $M$ remains constant, allowing it to be \textit{reordered Cholesky decomposed} \cite{ChDH08} in advance for improved computational efficiency. Specifically,  for each block $M_i$ of $M$, $i=1,\ldots,4$, can be decomposed by
\begin{equation} \label{eq:Chol}
U_i^\top U_i = P_i^\top M_i P_i,
\end{equation}
where $U_i$ is an upper triangular matrix and $P_i$ is a permutation matrix computed using the approximate minimum degree ordering \cite{AmDD04}. The following forward and backward substitutions can efficiently solve linear systems of the form $M_i\x = \r$:
\begin{subequations} \label{eq:LS}
\begin{equation} \label{eq:LS1}
U_i^\top \y = P_i^\top \r \,\text{ and }\, U_i\z = \y,
\end{equation}
with the final solution given by
\begin{equation} \label{eq:LS2}
\x = P_i\z. 
\end{equation}
\end{subequations}

To maximize the effectiveness of the search direction in the line search method \eqref{eq:line_search}, the ideal step length minimizes the objective function along this direction, i.e.,
\[
\alpha = \argmin_{x>0} \varphi(x) \equiv E_\rS( \f^{(k)} + x \p^{(k)}).
\]
However, computing the exact minimizer is computationally expensive due to the complexity of the energy functional $E_\rS$. Instead, we interpolate $\varphi(x)$ with a quadratic function $\varphi_q(x)$ and approximate the ideal step length with the minimizer of $\varphi_q(x)$ \cite[Chapter 3.5]{NoWr06}.

The coefficients of the quadratic function $\varphi_q(x) = ax^2 + bx + c$ are determined using the following three conditions:
\begin{subequations} \label{eq:step}
\begin{equation}
\begin{cases}
\varphi_q (0) = \varphi(0) = E_\rS( \f^{(k)})\\
\varphi_q' (0) = \varphi'(0) = {\g^{(k)}}^\top  \p^{(k)}\\
\varphi_q(\alpha_{k-1}) = \varphi(\alpha_{k-1}) = E_\rS( \f^{(k)} + \alpha_{k-1} \p^{(k)}),
\end{cases}
\end{equation} 
where $\alpha_{k-1}$ is the step length from the previous iteration, used as an initial guess. As a result, the step length $\alpha_k$ for the $k$th iteration is given by
\begin{equation}
\alpha_k = \argmin \varphi_q (x) = \frac{-b}{2a} = \frac{\varphi'(0)\,\alpha_{k-1}^2 }{2\, \big( \varphi(\alpha_{k-1}) - \varphi'(0) \alpha_{k-1} + \varphi(0) \big)}.
\end{equation}
\end{subequations}

It is worth noting that if $\varphi(\alpha)$ does not exhibit sufficient decrease, one can apply the quadratic interpolation again using the resulting step length. This approach tends to have a self-correcting effect. The computational procedure is summarized in Algorithm~\ref{alg:PCG}.

\begin{algorithm}[h]
\caption{Preconditioned nonlinear CG method for authalic map}
\label{alg:PCG}
\begin{algorithmic}[1]
\Require A simply connected open triangular mesh $\M$.
\Ensure A square area-preserving map $f$.
\State Initialize $k = 0$.
\State Compute initial map $\f_k$ via Algorithm~\ref{alg:FPM}.
\State Compute gradient $\g_k$ using \eqref{eq:grad_Es}.
\State Assemble preconditioner $M$ from \eqref{eq:precond} and compute Cholesky factorization by \eqref{eq:Chol}.
\State Solve $M \h_k = \g_k$ by \eqref{eq:LS} and set search direction $\p_k = -\h_k$.
\While{not converged}
    \State Determine step size $\alpha_k$ via \eqref{eq:step}.
    \State Update $\f_{k+1} \gets \f_k + \alpha_k \p_k$.
    \State Enforce boundary condition via \eqref{eq:bdry_cond}, depending on genus.
    \State Update gradient $\g_{k+1}$ and solve $M \h_{k+1} = \g_{k+1}$ by \eqref{eq:LS}.
    \State Update search direction $\p_{k+1} \gets -\h_{k+1} + \dfrac{\h_{k+1}^\top \g_{k+1}}{\h_k^\top \g_k} \p_k$.
    \State Update $k \gets k + 1$.
\EndWhile
\end{algorithmic}
\end{algorithm}

\subsection{Global convergence}
The preconditioned nonlinear CG method is globally convergent under the assumption that each step length $\alpha_k$ satisfies the \textit{strong Wolfe conditions} \cite[Chapter 3.1]{NoWr06}:
\begin{subequations} \label{eq:Wolfe}
\begin{align}
E_\rS(\f^{(k+1)}) - E_\rS(\f^{(k)})& \leq  c_1 \, \alpha_k {\g^{(k)}}^{\top} \p^{(k)}, \label{eq:Wolfe1} \\
|{\g^{(k+1)}}^{\top}  \p^{(k)}| & \leq c_2 \, | {\g^{(k)}}^{\top} \p^{(k)}|, \label{eq:Wolfe2}
\end{align}
\end{subequations}
where $0 < c_1 < c_2 < \frac{1}{2}$.
Under this assumption, the sequence generated by Algorithm~\ref{alg:PCG} satisfies the following convergence results:

\begin{theorem} \label{thm:convergence}
Suppose each step length in the preconditioned nonlinear CG method, Algorithm~\ref{alg:PCG} satisfies the strong Wolfe conditions \eqref{eq:Wolfe} with $0<c_1<c_2<\frac{1}{2}$. Then,
\[
\liminf_{k \rightarrow \infty} \| \g^{(k)} \|_{M^{-1}} = 0,
\]
provided $M$ is a symmetric positive definite matrix.
\end{theorem}
\begin{proof}
The stretch energy $E_\rS$ is bounded below due to $E_\rA \geq 0$. Moreover, since the map $\f$ lies in a compact domain $[-1, 1]^{2n}$, and $E_\rS$ is smooth, both the energy and its gradient are Lipschitz continuous on this domain. Given that the preconditioner $M$ is symmetric positive definite, the result follows directly from \cite[Appendix A]{LiYu24}.
\end{proof}

Theorem~\ref{thm:convergence} implies the existence of a subsequence $\{\g^{(k_\ell)} \}_{\ell=1}^\infty$ that converges to zero. Consequently, for any prescribed tolerance $\varepsilon > 0$, Algorithm~\ref{alg:PCG} is guaranteed to terminate under the stopping criterion $\| \g^{(k)} \|_{M^{-1}} < \varepsilon$ for some $k$.

\subsection{Enforcing the bijectivity}
To guarantee the bijectivity of the resulting parameterizations, we apply an overlap correction step based on convex combination mappings, since such mappings are guaranteed to be bijective when the boundary is mapped to a convex region \cite[Theorem 4.1]{Floa03}.

A \textit{convex combination map} is defined so that each interior vertex is expressed as a convex combination of its neighbors. That is, for every interior vertex $v_i$, there exist weights $\lambda_{i,j} > 0$ for $v_j \in \N(v_i)$, the neighborhood of $v_i$, such that
\[
f(v_i) = \sum_{j \in \N (v_i)} \lambda_{i,j} f(v_j), ~~~  \sum_{j \in \N (v_i)} \lambda_{i,j} = 1.
\]
Equivalently, this condition can be expressed as $[L(f)\,\f^s]_\I = \0$ for some Laplacian matrix $L(f)$ with positive weights, since the construction of the Laplacian implies that
$$
f_i = \sum_{j \in \N_i} \frac{-L(f)_{i,j}}{L(f)_{i,i}} \, f_j, ~~~ \sum_{j \in \N_i} \frac{-L(f)_{i,j}}{L(f)_{i,i}}  = 1.
$$

To realize such a map, we construct the Laplacian using \textit{mean value weights} \cite{Floa03b}, defined by
\begin{equation}
L_\mathrm{M}(f)_{i,j} = 
\begin{cases}
    - \sum_{[v_i, v_j, v_k] \in \F} \frac{\tan (\gamma_{j,k}^i(f) /2)}{\|f_i - f_j\|_2}  & \text{if $[v_i, v_j] \in \E$,} \\
    -\sum_{\ell \neq i} L_\mathrm{M}(f)_{i, \ell}  & \text{if $j = i$,} \\
    0   &\text{otherwise,}
\end{cases}
\label{eq:L_M}
\end{equation}
where $\gamma_{j,k}^i(f)$ is the angle opposite the edge $[f_j, f_k]$ at the vertex $f_i$ on the mapped surface $f(\M)$ (see Figure~\ref{fig:mean value}).
We then update the interior maps by solving the linear systems
\begin{equation}
L_\mathrm{M}(f)_{\I, \I} \f^s_{\I} = - L_\mathrm{M}(f)_{\I, \B} \f^s_{\B},~~~ s = 1,2.
\label{eq:MVT_correct}
\end{equation}
Because $0 < \gamma_{j,k}^i(f) < \pi$, the mean value weights satisfy $\tan(\gamma_{j,k}^i(f)/2) > 0$, ensuring that the resulting mapping is a convex combination map, and thus bijective whenever the boundary is convex \cite[Theorem 4.1]{Floa03}.

\begin{figure}[]
\centering
\begin{tikzpicture}[black,thick,scale=1.5]
\coordinate (v_i) at (0,0);
\coordinate (v_j) at (0,2);
\coordinate (v_l) at (2,1);
\coordinate (v_k) at (-2,1);
\filldraw[camel!60] (v_i) -- (v_j) -- (v_k);
\filldraw[camel!60] (v_i) -- (v_j) -- (v_l);
\pic[draw, ->, "$\gamma_{j, k}^i (f)$", angle eccentricity=1.95, angle radius=0.6cm]{angle = v_j--v_i--v_k};
\pic[draw, <-, "$\gamma_{j, \ell}^i (f)$", angle eccentricity=1.95, angle radius=0.6cm]{angle = v_l--v_i--v_j};
\draw{
(v_i) -- (v_j) -- (v_k) -- (v_i) -- (v_l) -- (v_j)
};
\tikzstyle{every node}=[circle, draw, fill=lightgoldenrodyellow!30,
                        inner sep=1pt, minimum width=2pt]
\draw{
(0,0) node{$f_i$}
(0,2) node{$f_j$}
(2,1) node{$f_\ell$}
(-2,1) node{$f_k$}
};
\end{tikzpicture}
\caption{An illustration of the mean-value weight \cite{Floa03b} defined on $f(\M)$.}
\label{fig:mean value}
\end{figure}
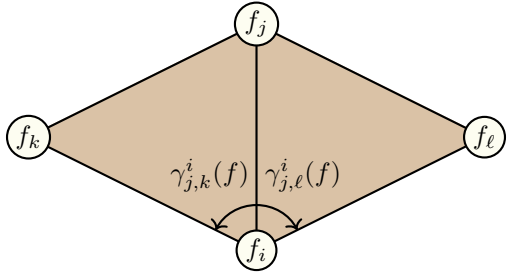

\begin{figure}[h]
\centering
\resizebox{\textwidth}{!}{
\begin{tabular}{ccccccc}
\specialrule{.2em}{.1em}{.1em}
\multicolumn{3}{c}{David Head} && \multicolumn{3}{c}{Bull} \\
\multicolumn{3}{c}{$\F(\M) = 21,338$ $\V(\M) = 10,671$} && \multicolumn{3}{c}{$\F(\M) = 34,504$ $\V(\M) = 17,254$} \\
\includegraphics[height=3cm]{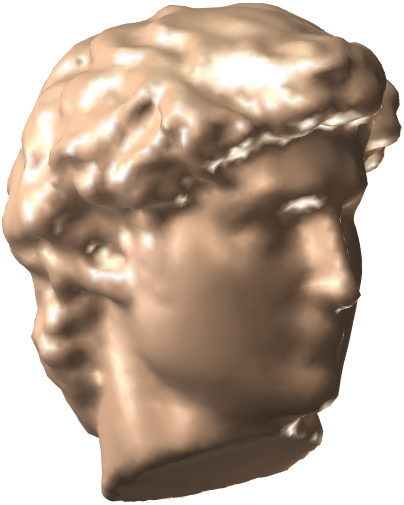} &
\includegraphics[height=3cm]{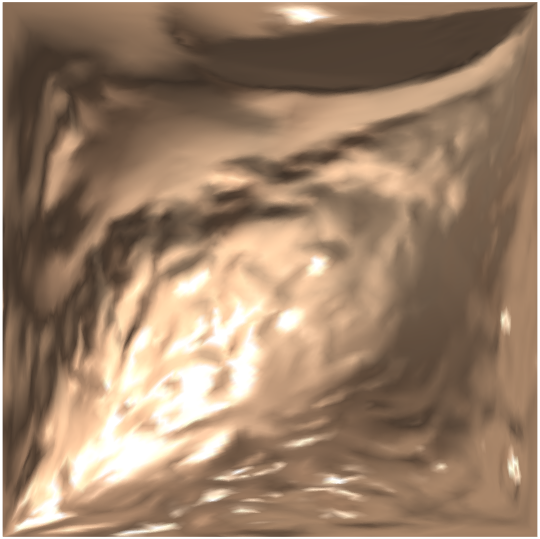} &
\includegraphics[height=3cm]{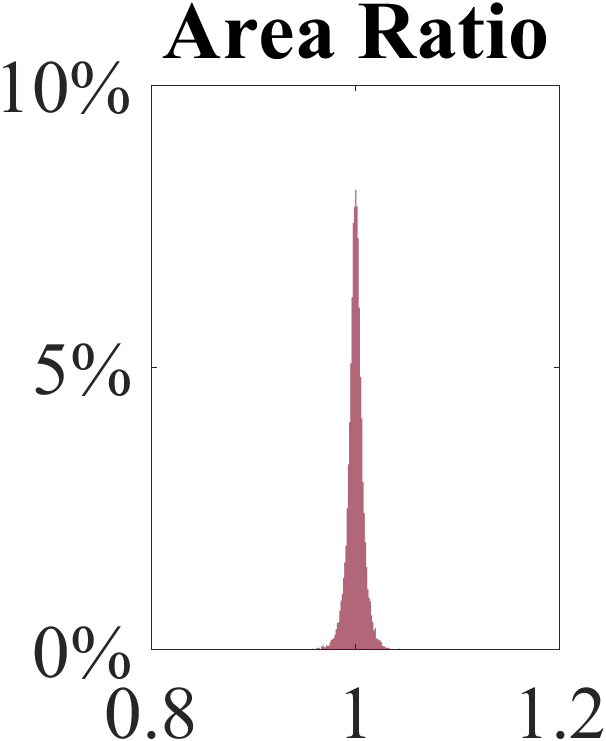} &&
\includegraphics[height=3cm]{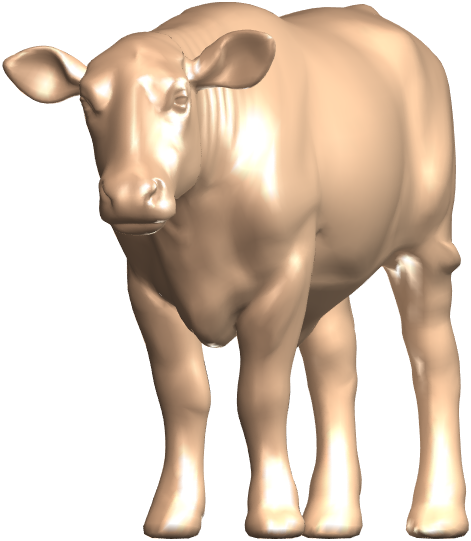} &
\includegraphics[height=3cm]{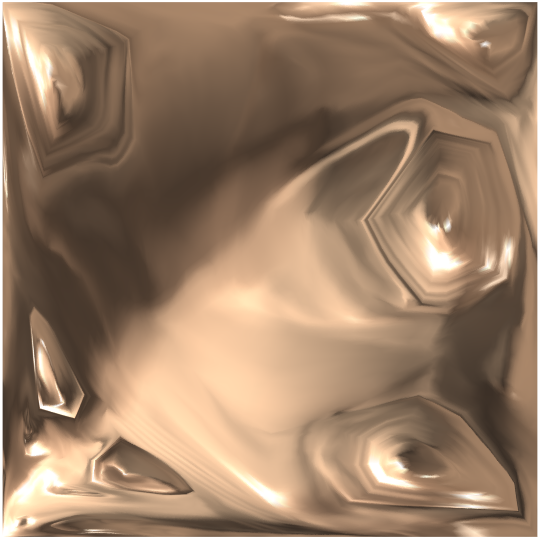} &
\includegraphics[height=3cm]{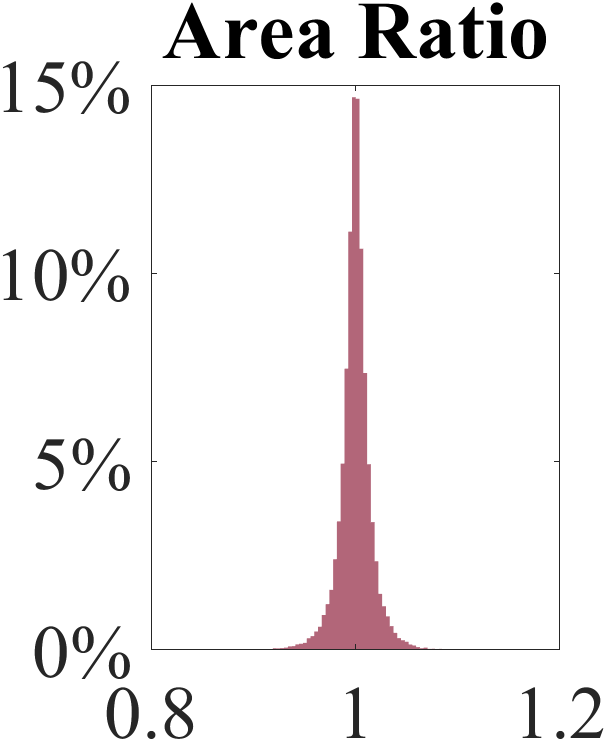} \\
\Cline{1pt}{1-3}\Cline{1pt}{5-7}
\multicolumn{3}{c}{Bunny} && \multicolumn{3}{c}{Human Brain} \\
\multicolumn{3}{c}{$\F(\M) = 111,364$ $\V(\M) = 55,916$} && \multicolumn{3}{c}{$\F(\M) = 120,000$ $\V(\M) = 60,155$} \\
\includegraphics[height=3cm]{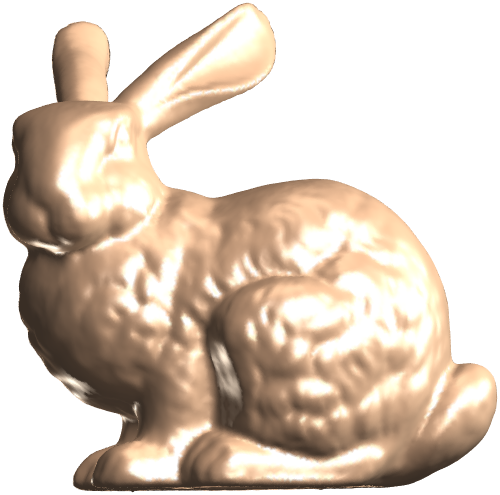} &
\includegraphics[height=3cm]{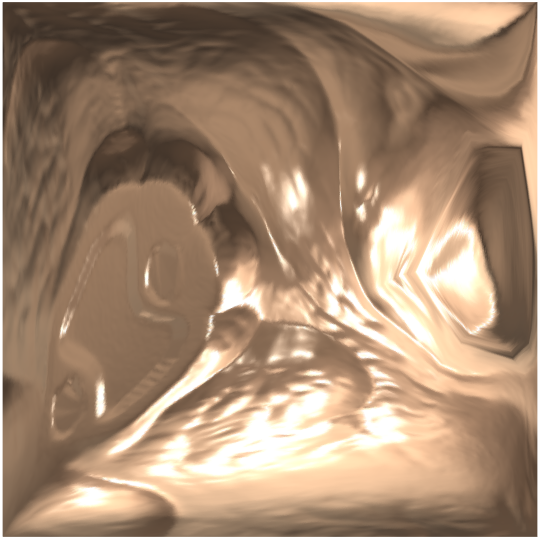} &
\includegraphics[height=3cm]{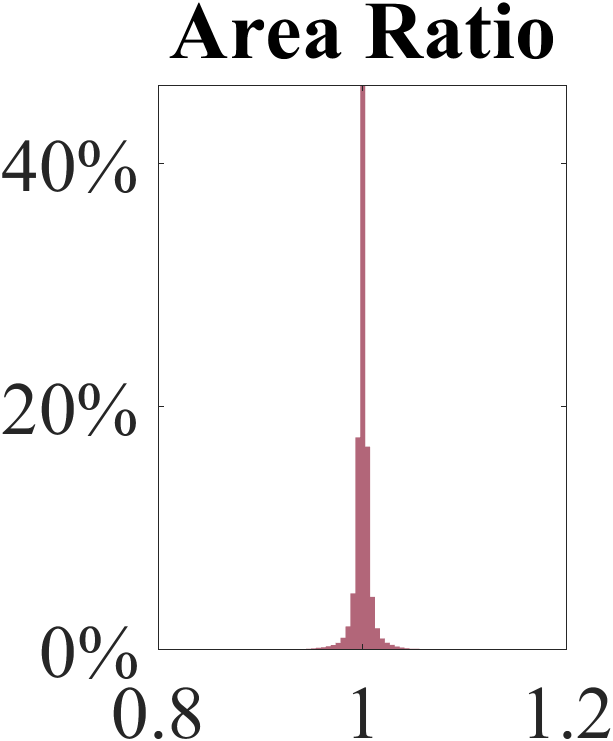} &&
\includegraphics[height=3cm]{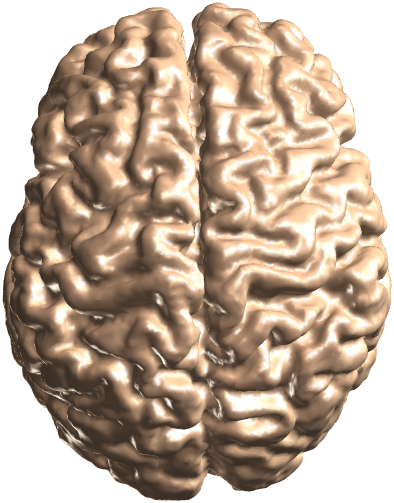} &
\includegraphics[height=3cm]{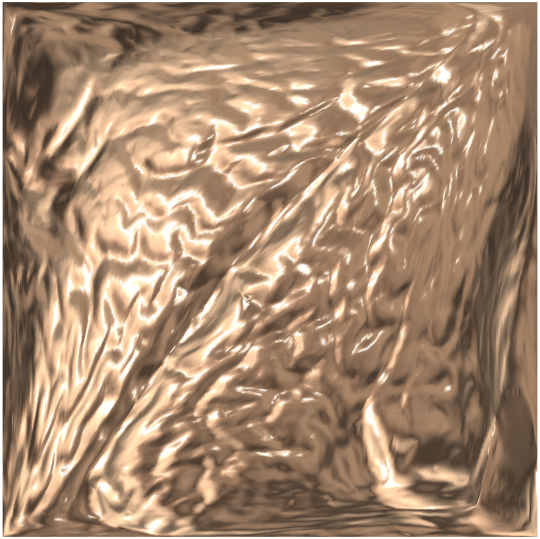} &
\includegraphics[height=3cm]{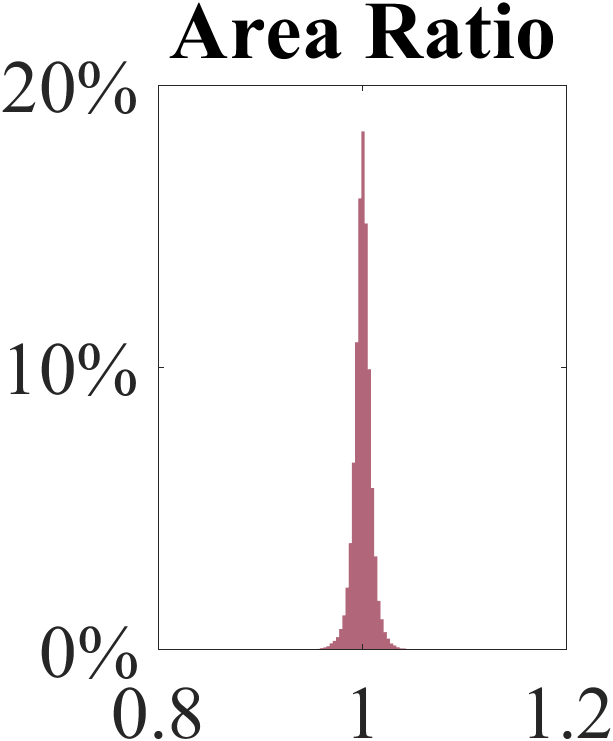} \\
\specialrule{.2em}{.1em}{.1em}
\end{tabular}
}
\caption{Genus-zero triangular meshes and their area-preserving parameterizations computed by our method, together with histograms of the area ratio \eqref{eq:Area_dist}.}
\label{fig:MeshModel_g0}
\end{figure}

\begin{figure}[h]
\centering
\resizebox{\textwidth}{!}{
\begin{tabular}{ccccccc}
\specialrule{.2em}{.1em}{.1em}
\multicolumn{3}{c}{Vertebrae} && \multicolumn{3}{c}{Kitten} \\
\multicolumn{3}{c}{$\F(\M) = 16,420$ $\V(\M) = 8,297$} && \multicolumn{3}{c}{$\F(\M) = 20,000$ $\V(\M) = 10,097$} \\
\includegraphics[height=3cm]{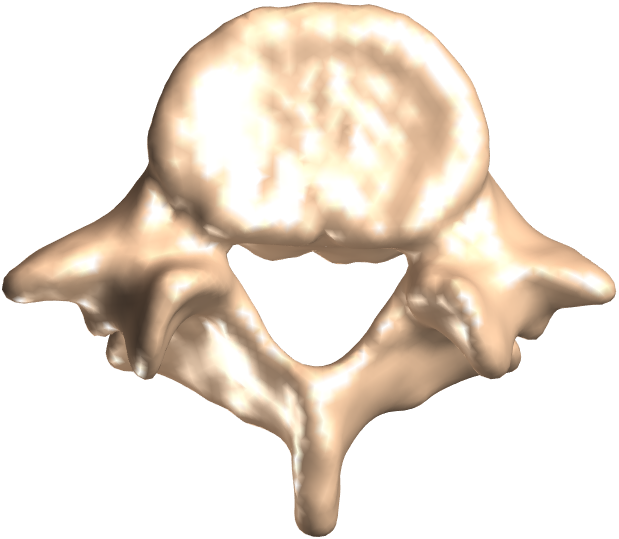} &
\includegraphics[height=3cm]{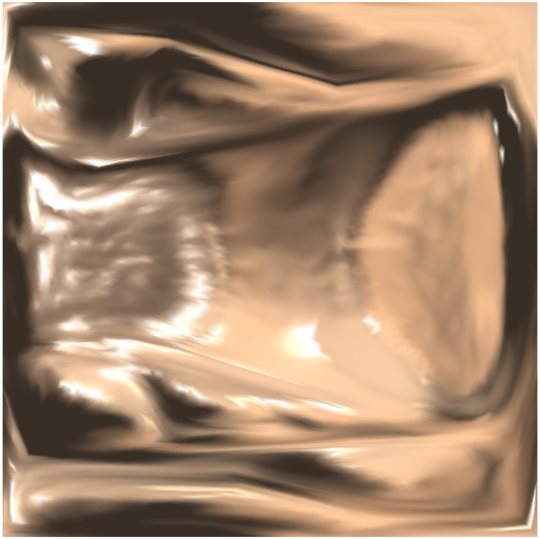} &
\includegraphics[height=3cm]{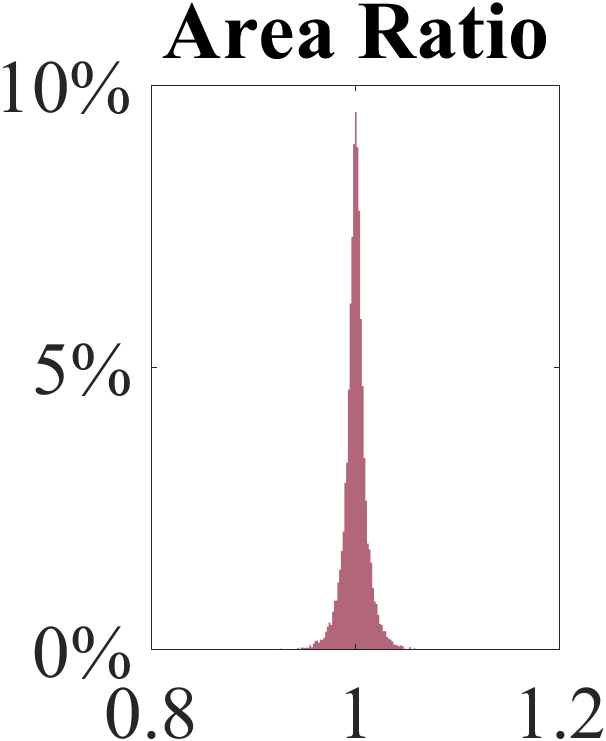} &&
\includegraphics[height=3cm]{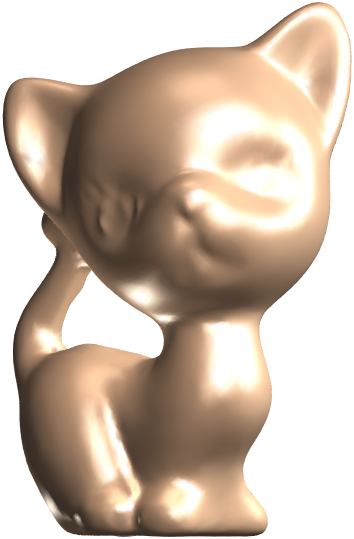} &
\includegraphics[height=3cm]{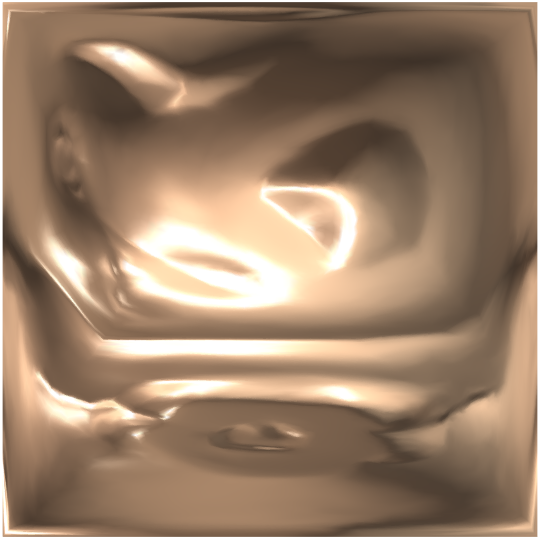} &
\includegraphics[height=3cm]{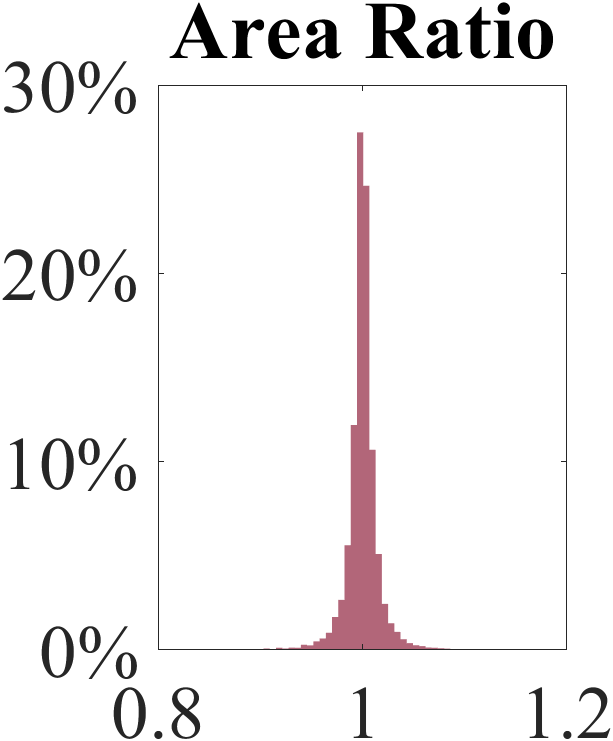} \\
\Cline{1pt}{1-3}\Cline{1pt}{5-7}
\multicolumn{3}{c}{Rocker Arm} && \multicolumn{3}{c}{Chess Horse} \\
\multicolumn{3}{c}{$\F(\M) = 20,088$ $\V(\M) = 10,165$} && \multicolumn{3}{c}{$\F(\M) = 46,016$ $\V(\M) = 23,137$} \\
\includegraphics[height=3cm]{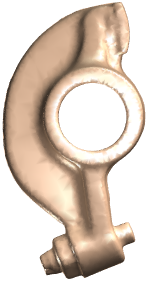} &
\includegraphics[height=3cm]{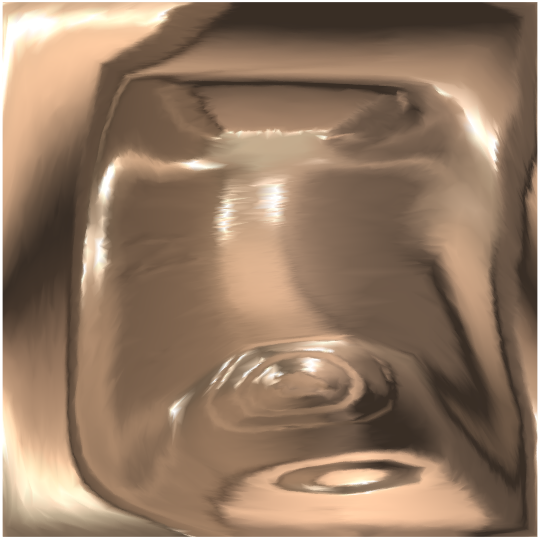} &
\includegraphics[height=3cm]{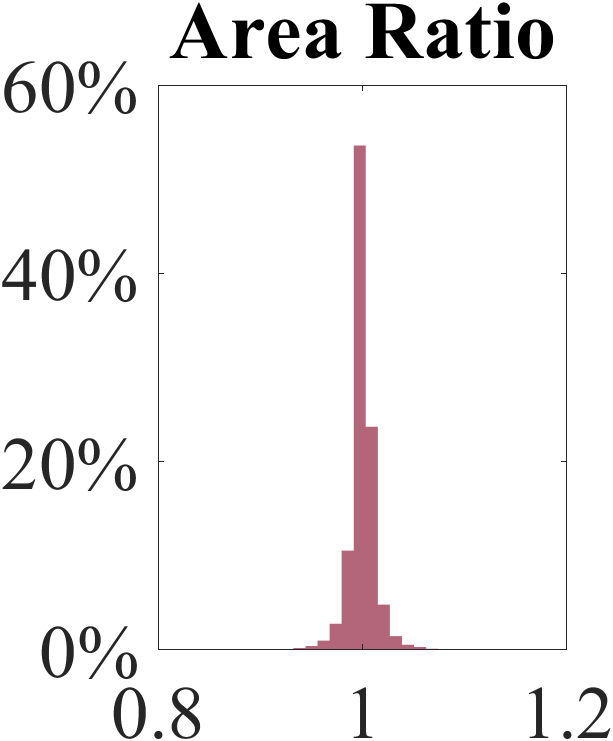} &&
\includegraphics[height=3cm]{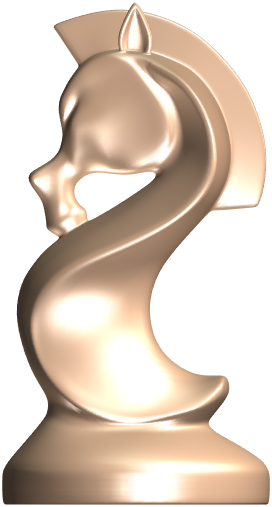} &
\includegraphics[height=3cm]{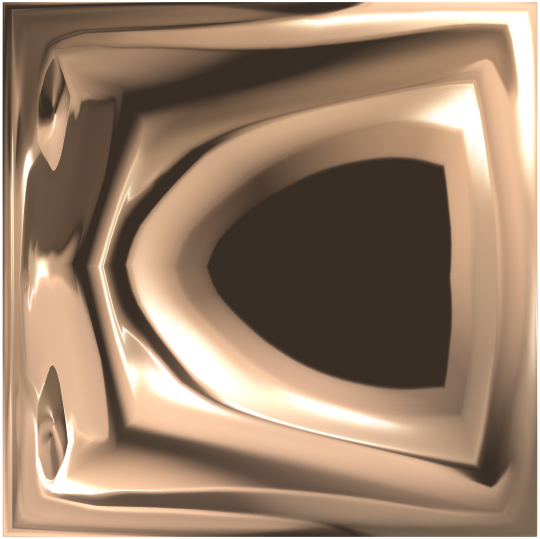} &
\includegraphics[height=3cm]{Rocker_arm_Hist.png} \\
\specialrule{.2em}{.1em}{.1em}
\end{tabular}
}
\caption{Genus-one triangular meshes and their area-preserving parameterizations computed by our method, together with histograms of the area ratio \eqref{eq:Area_dist}.}
\label{fig:MeshModel_g1}
\end{figure}

\section{Numerical experiments} \label{sec:6}
In this section, we present numerical experiments to demonstrate the effectiveness of the proposed method. The datasets used in our experiments are adapted from publicly available sources, including the Stanford 3D Scanning Repository \cite{Stanford} and Sketchfab \cite{Sketchfab}. All computations were performed in MATLAB on a laptop equipped with an AMD Ryzen 9 5900HS processor and 32 GB of RAM.

\subsection{Square area-preserving parameterizations}
To quantify the area distortion, we define the area ratio for a triangle $\tau$ as
\begin{equation}
    \cR_\mathrm{area}(\tau) =  \bigg| \frac{|f(\tau)| / \A(f) }{|\tau| / |\M|} \bigg|,
\label{eq:Area_dist}
\end{equation}
where $|f(\tau)|$ and $|\tau|$ denote the areas of $\tau$ in the parameter domain and in the original mesh $\M$, respectively, and $\A(f)$ is the total area of the parameter domain. For a perfectly area-preserving mapping, the mean and standard deviation of $\cR_\mathrm{area}$ over all mesh faces satisfy
\[
\mean_{\tau \in \F(\M)} \cR_\mathrm{area}(\tau) = 1, ~~\mbox{and}~~   \SD_{\tau \in \F(\M)} \cR_\mathrm{area}(\tau) = 0.
\]

Figures~\ref{fig:MeshModel_g0} and \ref{fig:MeshModel_g1} show the benchmark triangular mesh models with genus-zero and genus-one, their corresponding parameterizations, and the histograms of area ratios. In each case, the parameterization is obtained by first applying Algorithm~\ref{alg:slice_mesh} to slice the closed mesh, followed by Algorithm~\ref{alg:PCG} to map it onto a unit square domain under the boundary conditions \eqref{eq:bdry_cond} and ensuring bijectivity through the linear systems in \eqref{eq:MVT_correct}.

The numerical results for all benchmark models are summarized in Table~\ref{tab:square_map}. The resulting mappings exhibit low $E_\rA$ values, indicating good area preservation, since $E_\rA$ reflects the variance of the area-weighted area ratios, as stated in Theorem~\ref{thm:Ea_var}. This is further supported by the small standard deviation of the area ratio $\cR_\mathrm{area}$, which confirms Theorem~\ref{thm:Ea_unwvar} that the variance of $\cR_\mathrm{area}$ is bounded by $E_\rA$. All examples were completed in under 15 seconds. Furthermore, no folded triangular faces were observed following the correction in \eqref{eq:MVT_correct}, confirming the bijectivity of the resulting mappings.

In summary, the proposed method efficiently produces highly area-preserving mappings while ensuring bijectivity.

\begin{table}[]
\centering
\caption{ 
Numerical results of square parameterization obtained by Algorithm~\ref{alg:PCG} for all benchmark models, sliced using Algorithm~\ref{alg:slice_mesh}. 
$E_\rA$: authalic energy functional \eqref{eq:Ea}. 
$\cR_\mathrm{area}$: area ratios \ref{eq:Area_dist}. 
$\#$Foldings: number of folded triangles.
$\#$Iterations: total number of iterations.
}
\label{tab:square_map}
\begin{tabular}{lrccccc}
\specialrule{.2em}{.1em}{.1em}
\multirow{2}{*}{Model name}   & Time     & \multirow{2}{*}{$E_\rA$}  & \multicolumn{2}{c}{$\cR_\mathrm{area}$} &  $\#$Fold- &  $\#$Iter-\\ 
& {(secs.)}  &        &  mean  & SD   & ings$^*$  & ations$^\dagger$   \\
\hline 
David Head   & 1.07  &$6.15 \times 10^{-5}$ &$1.00 $ & $8.07 \times 10^{-3}$  & 0 & 52  \\
Bull         & 4.32  &$1.57 \times 10^{-4}$ &$1.00 $ & $1.63 \times 10^{-2}$  & 0 & 172  \\
Bunny        & 12.47 &$1.17 \times 10^{-4}$ &$1.00 $ & $1.08 \times 10^{-2}$  & 0 & 116   \\
Human Brain  & 7.94  &$7.18 \times 10^{-5}$ &$1.00$ & $8.90 \times 10^{-3}$  & 0 & 63 \\
Vertebrae    & 2.32  &$1.39 \times 10^{-4}$ &$1.00$ & $1.21 \times 10^{-2}$  & 0 & 120  \\
Kitten       & 3.08  &$1.74 \times 10^{-4}$ &$1.00$ & $1.75 \times 10^{-2}$  & 0 & 148  \\
Rocker Arm   & 2.14  &$8.25 \times 10^{-5}$ &$1.00$ & $1.29 \times 10^{-2}$  & 0 & 83  \\
Chess Horse  & 10.08 &$7.80 \times 10^{-4}$ &$1.00$ & $3.89 \times 10^{-2}$  & 1 $\rightarrow$ 0 & 200  \\
\specialrule{.2em}{.1em}{.1em}
\multicolumn{7}{l}{\small  $^*$: Before (left) and after (right) correction \eqref{eq:MVT_correct}.  } \\
\multicolumn{7}{l}{\small  $^\dagger$: Stopping criteria: energy deficit less than $10^{-6}$ or reach 200 iterations. } \\
\end{tabular}
\end{table}

\subsection{Relationship between authalic energy and area ratio variance}
\label{sec:6.2}
As established in Theorem~\ref{thm:Ea_var}, the authalic energy $E_\rA$ in \eqref{eq:Ea} represents the area-weighted variance of the area ratios. In practice, most literature measures local area distortion using the standard deviation of the area ratio $\cR_\mathrm{area}$ in \eqref{eq:Area_dist}. In this section, we numerically examine the relation between $E_\rA$ and the variance of $\cR_\mathrm{area}$.

Figure~\ref{fig:Ea_vs_AreaRatio} shows the variance of $\cR_\mathrm{area}$ and $E_\rA$ over 50 iterations of Algorithm~\ref{alg:PCG} for all benchmark models. The iterations include Algorithm~\ref{alg:FPM} for the initial mappings, which causes the small twists around the 10th iteration in each curve. 
For almost all tested triangular meshes, our method exhibits consistent convergence patterns in both $E_\rA$ and the variance of $\cR_\mathrm{area}$. Moreover, $E_\rA$ shows an almost linear correlation with the variance of $\cR_\mathrm{area}$ throughout the iterations. The Chess Horse model is a notable exception, likely due to its more complicated geometry, yet it still demonstrates a strong positive relationship between $E_\rA$ and the variance of $\cR_\mathrm{area}$.

In summary, these results confirm that minimizing $E_\rA$ effectively reduces local area distortion.

\begin{figure}[]
\centering
\resizebox{\textwidth}{!}{
\begin{tabular}{ccc}
\includegraphics[height=4cm]{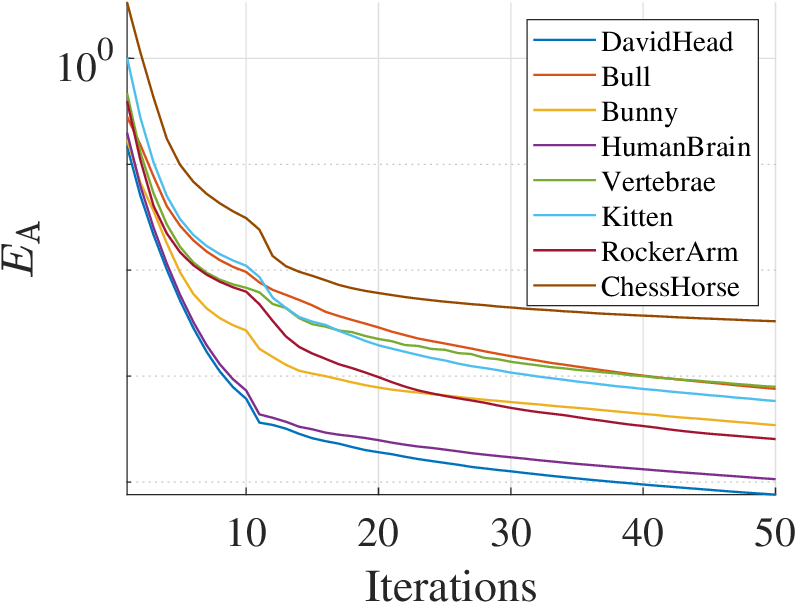} &
\includegraphics[height=4cm]{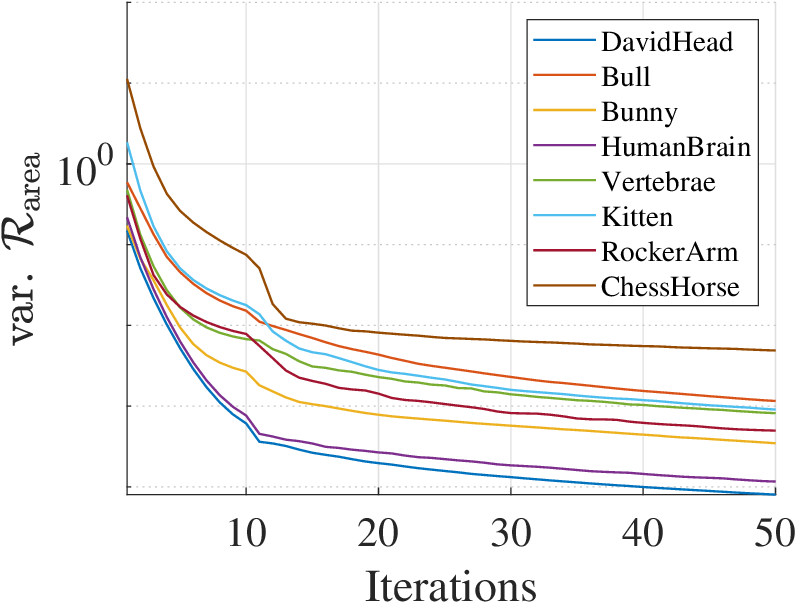} 
\includegraphics[height=4cm]{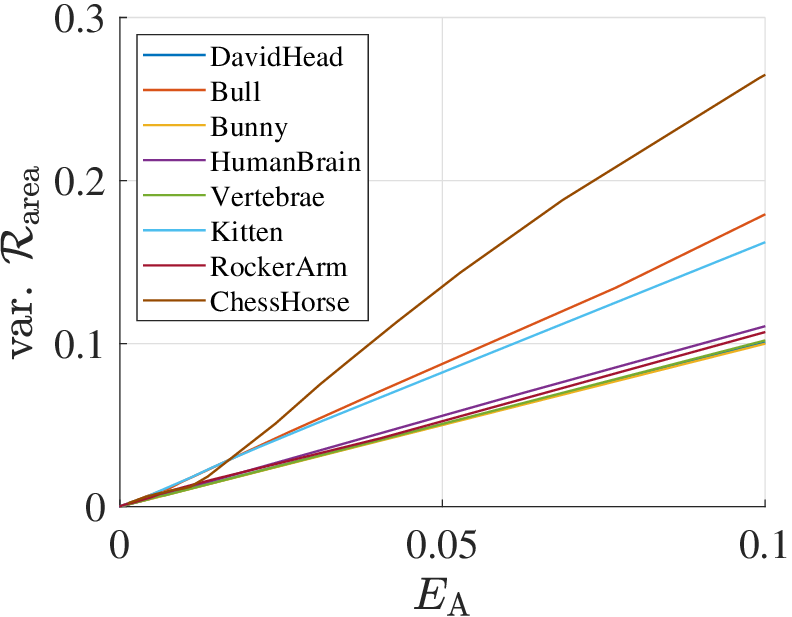} 
\end{tabular}
}
\caption{Relationship between $E_\rA$ in \eqref{eq:Ea} and the variance of $\cR_\mathrm{area}$ in \eqref{eq:Area_dist} during iterations of Algorithm~\ref{alg:PCG}.}
\label{fig:Ea_vs_AreaRatio}
\end{figure}

\begin{figure}[]
\center
\resizebox{\textwidth}{!}{
\begin{tabular}{ccccccc}
\includegraphics[height=3cm]{DavidHead_mesh.png} &
\raisebox{1.2cm}{$\longrightarrow$} & 
\includegraphics[height=2.7cm]{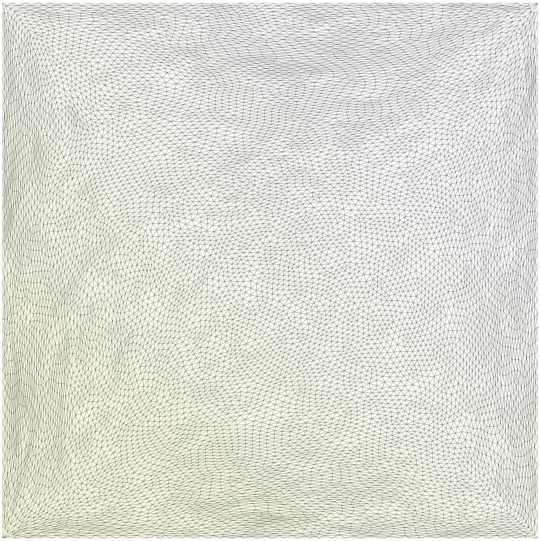} & 
\raisebox{1.2cm}{$\longrightarrow$} &
\begin{overpic}[width=3cm]{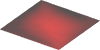}
    \put(0,25){\includegraphics[width=3cm]{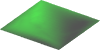}}
    \put(0,50){\includegraphics[width=3cm]{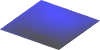}}
\end{overpic} & 
\raisebox{1.2cm}{$\longrightarrow$} & 
\includegraphics[height=2.7cm]{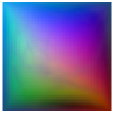}\\
Triangular Mesh & & Parameterization & & Interpolation  & & Geometry image
\end{tabular}
}
\caption{The process of generating a geometry image from a triangular mesh.}
\label{fig:GeoImg}
\end{figure}

\begin{figure}[]
\center
\resizebox{\textwidth}{!}{
\begin{tabular}{ccccccc}
\includegraphics[height=2.7cm]{GeoImg.png} &
\raisebox{1.2cm}{$\longrightarrow$} & 
\includegraphics[height=2.7cm]{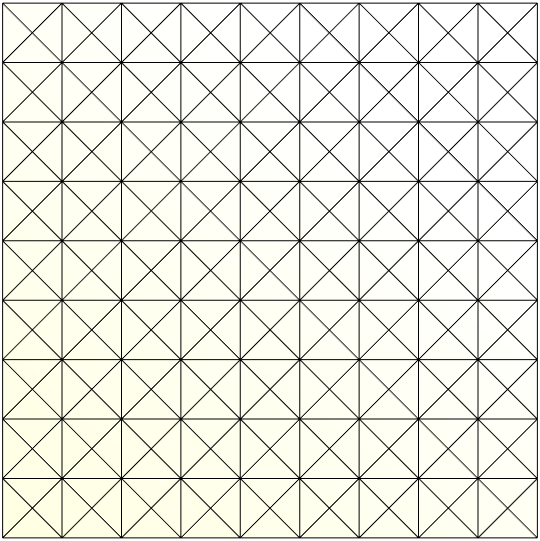} & 
\raisebox{1.2cm}{$\longrightarrow$} &
\begin{overpic}[width=3cm]{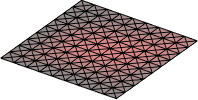}
    \put(0,25){\includegraphics[width=3cm]{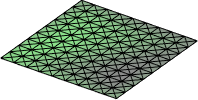}}
    \put(0,50){\includegraphics[width=3cm]{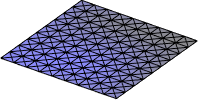}}
\end{overpic} & 
\raisebox{1.2cm}{$\longrightarrow$} & 
\includegraphics[height=3cm]{DavidHead_Remesh.png}\\
Geometry image & & Planner mesh & & Interpolation & & Triangular mesh
\end{tabular}
}
\caption{The process of generating a triangular mesh from a geometry image. }
\label{fig:GeoImgRe}
\end{figure}

\section{Application in geometry image} \label{sec:7}
A geometry image \cite{GuGH02} is an RGB image representation of a surface, where each pixel encodes either a vertex position or a vertex normal. Figure~\ref{fig:GeoImg} illustrates the generation process. First, we compute a square parameterization of the surface. The $x$, $y$, and $z$ coordinates of the vertices are then interpolated onto a uniform lattice grid and stored in the red, green, and blue channels of the image, respectively. The original surface can thus be recovered from the geometry image by reversing this process, as shown in Figure~\ref{fig:GeoImgRe}. In particular, we reconstruct the mesh by inserting a center point into each quadrilateral element of the geometry image. Vertex positions are obtained by interpolating the values stored in the image, while the triangular connectivity is defined by the triangulation on the image.

The parameterization used to generate geometry images has a significant effect on the reconstructed surface. In fact, when the square parameterization aligns with the mesh used in the geometry image, the reconstructed triangular surface has identical vertex positions and triangular connectivity as the original surface. Thus, the ideal case is when the triangular mesh of the square parameterization closely matches the mesh obtained by subdividing quads via their face centers.

To approach this ideal setting, we first compute parameterizations with constant face areas by introducing an area measure into our energy minimization method, as described in Subsection~\ref{sec:7.1}. Then, we correct triangles with severe angular distortion using a truncation strategy for the Beltrami coefficient, based on quasi-conformal theory \cite{LyLC24}, as detailed in Subsection~\ref{sec:7.2}. Finally, we enhance the mesh quality of the reconstructed surface through remeshing and smoothing, described in Subsection~\ref{sec:7.3}, and present the corresponding numerical results in Subsection~\ref{sec:7.4}.

\subsection{Constant face area map via measure preservation}
\label{sec:7.1}
To construct a mapping with constant face areas, we introduce an area measure $\rho : \F \to \mathbb{R}_+$ and generalize the stretch energy in \eqref{eq:Es} as
\[
E_{\rS_\rho}(f) = \sum_{\tau \in \F}\frac{|f(\tau)|^2}{\rho(\tau)} = \frac{1}{2} \sum_{s = 1}^{2} {\f^s}^\top L_{\rS_\rho}(f) \, \f^s,
\]
where the weighted Laplacian matrix is defined by
\begin{equation*} 
[L_{\rS_\rho}(f)]_{i, j} = 
\begin{cases}
    \displaystyle -\sum_{\tau = [v_i,v_j,v_k] } \frac{\cot(\theta_{i, j}^k(f))}{2} \frac{|f(\tau)|}{\rho(\tau)}   &\text{if $[v_i, v_j] \in \mathcal{E}$,} \\[0.5em]
    \displaystyle -\sum_{\ell \neq i} [L_{\rS_\rho}(f)]_{i, \ell}     &\text{if $j=i$,}\\
    0       &\text{otherwise}.
\end{cases}
\end{equation*}
By Theorem~\ref{thm:Ea_var}, replacing $|\tau|$ with $\rho(\tau)$ shows that, when $\A(f) = |\M|$, the functional $E_{\rS_\rho}(f) - |\M|$ corresponds to the variance of $\tfrac{|f(\tau)|}{\rho(\tau)}$. Therefore, a mapping with constant face areas can be obtained by applying the same procedure as Algorithm~\ref{alg:PCG} with $\rho(\tau) = 1$.

However, constant face area maps typically exhibit many thin triangles, which may span across multiple pixels of a geometry image. This can impair the reconstruction of fine geometric details, particularly for surfaces with complicated shapes. To address this issue, we correct triangles with severe angular distortion in the resulting map based on quasi-conformal theory, as described in Subsection~\ref{sec:7.2}.

\subsection{Correction by quasi-conformal map}
\label{sec:7.2}
A \textit{quasi-conformal mapping} $f:\mathbb{C}\to\mathbb{C}$ is a complex-valued function that satisfies the Beltrami equation
\[
\frac{\partial f}{\partial \bar{z}} = \mu(z) \frac{\partial f}{\partial z} ,
\]
where $\mu$ is a complex-valued function with $\|\mu \|_\infty < 1$. The function $\mu$, called the \textit{Beltrami coefficient} of $f$, measures the local angular distortion. In particular, if $\mu(z)=0$, then $f$ satisfies the Cauchy--Riemann equations at $z$, and hence is conformal in a neighborhood of $z$.

Let $f_h$ denote the harmonic map (Algorithm~\ref{alg:FPM} with $k=0$) and $f_c$ the constant face area map. By the measurable Riemann mapping theorem \cite{GaLa00}, a quasi-conformal map $\phi: f_h \to f_c$ is uniquely determined by its Beltrami coefficient, denoted $\mu_\phi$. Based on this, we characterize $f_c$ by $\mu_\phi$ and then reduce excessive distortion by truncating $\mu_\phi$ according to a prescribed threshold $\delta$. The truncated Beltrami coefficient is defined as
\[
\widetilde{\mu}_\phi(z) =
\begin{cases}
    \mu_\phi(z), & \text{if } |\mu_\phi(z)| < \delta, \\
    \frac{\delta }{|\mu_\phi(z)|} \mu_\phi(z), & \text{otherwise},
\end{cases}
\]
and the corrected mapping is then reconstructed using the linear Beltrami solver \cite{LuLW13} with respect to $\widetilde{\mu}_\phi$. In practice, we adopt $\delta = 0.8$.

\subsection{Remeshing and smoothing}
\label{sec:7.3}
To further improve the mesh quality of surfaces reconstructed from geometry images, we apply a remeshing and smoothing procedure based on harmonic parameterizations. The complete remeshing procedure is illustrated in Figure~\ref{fig:GeoImg_remesh}.

By the uniformization theorem, the conformal equivalence classes of genus-zero and genus-one surfaces are the sphere and the plane, respectively. For genus-zero surfaces, we first glue the cut surface back into a closed mesh and then compute a spherical harmonic map \cite{ChLL15}. Remeshing is performed by applying Delaunay triangulation on the stereographic plane twice: once for all vertices, and once for boundary vertices only, to reconstruct a closed mesh. For genus-one surfaces, we compute a rectangular harmonic parameterization with the aspect ratio chosen to minimize the Dirichlet energy. The triangular connectivity is then obtained by Delaunay triangulation, followed by gluing to form the final mesh.

Finally, we smooth the mesh using mean curvature flow along its surface normals \cite{DeMe99, Bely06}, which is solved using the implicit Euler method. The MATLAB implementation is available in the function $\mathtt{smoothpatch}$ \cite{smoothpatch}.

\begin{figure}[]
\center
\resizebox{\textwidth}{!}{
\begin{tabu}{ccccccc}
\includegraphics[height=2.5cm]{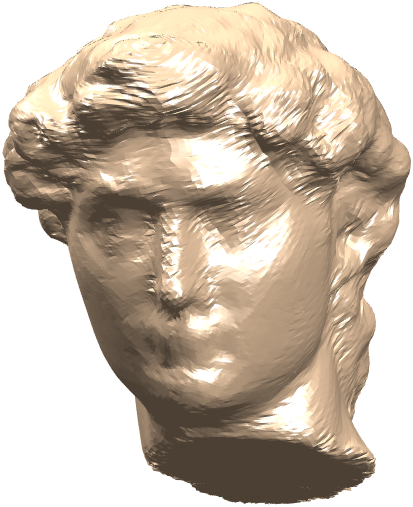} & 
\raisebox{0.8cm}{$\longrightarrow$} & 
\includegraphics[height=2.5cm]{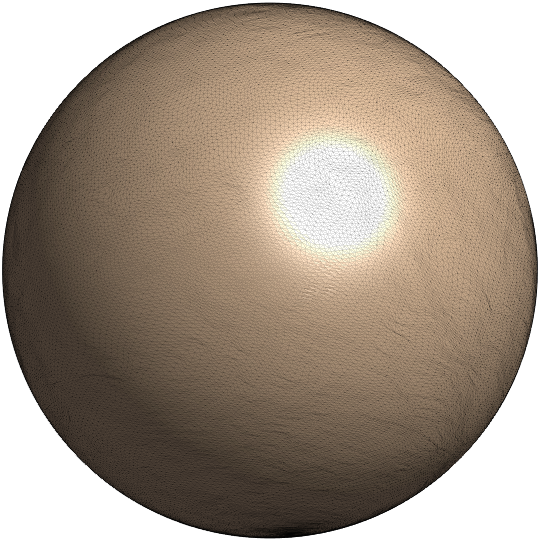} \\
Glue mesh & & Harmonic spherical map \\
$\uparrow$ & & $\downarrow$\\
\includegraphics[height=2.5cm]{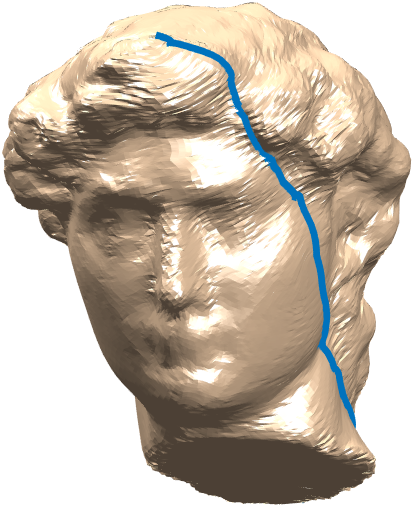} & & 
\includegraphics[height=2.5cm]{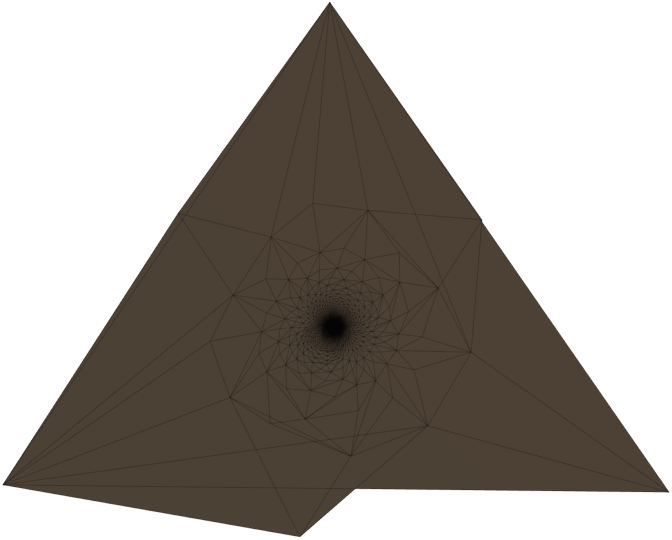} &
\raisebox{0.8cm}{$\longrightarrow$} & 
\includegraphics[height=2.5cm]{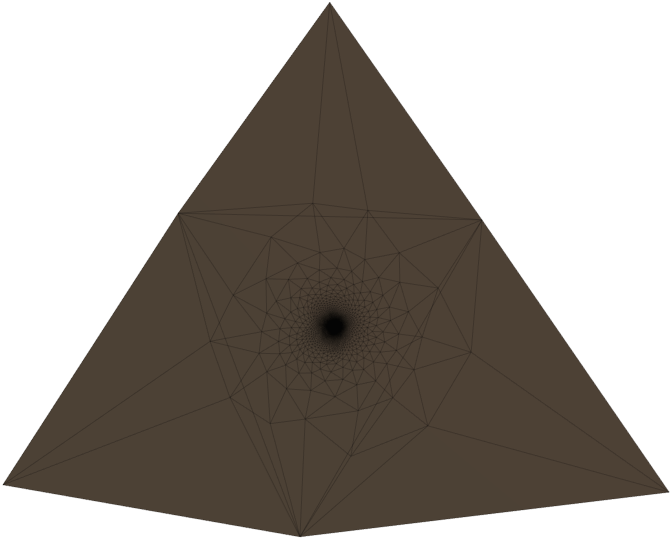} &
\raisebox{0.8cm}{$\longrightarrow$} & 
\includegraphics[height=2.5cm]{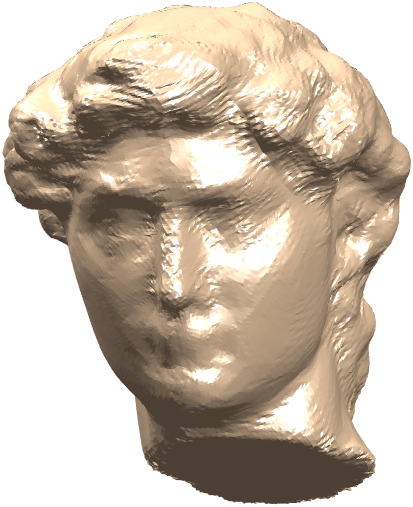}\\
Cut mesh & & Stereographic projection & & Triangulation & & Remeshing \\
\multicolumn{7}{c}{\upbracefill}\\
\multicolumn{7}{c}{Remeshing for genus-zero simplicial surface} \\ 
 & & & & 
\includegraphics[height=2.5cm]{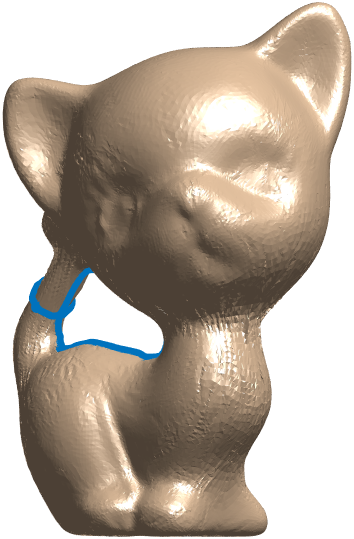} \\
 & & & & Remeshing \\
 & & & & $\uparrow$ & $\searrow$ \\
\includegraphics[height=2.5cm]{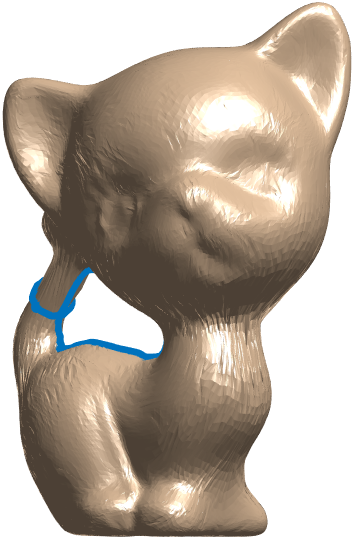} & 
\raisebox{0.8cm}{$\longrightarrow$} & 
\includegraphics[height=2.5cm]{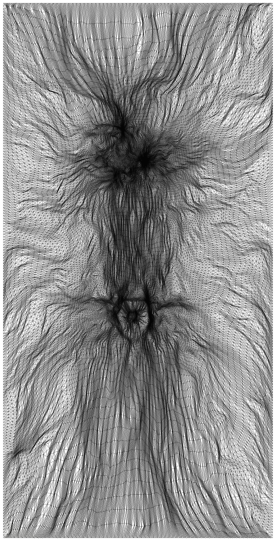} &
\raisebox{0.8cm}{$\longrightarrow$} & 
\includegraphics[height=2.5cm]{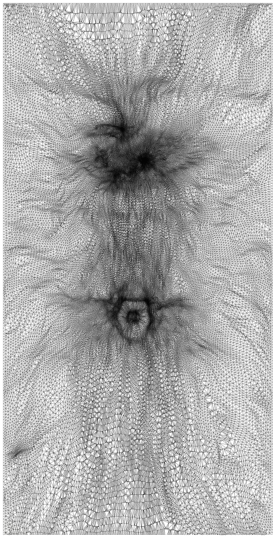} & &
\includegraphics[height=2.5cm]{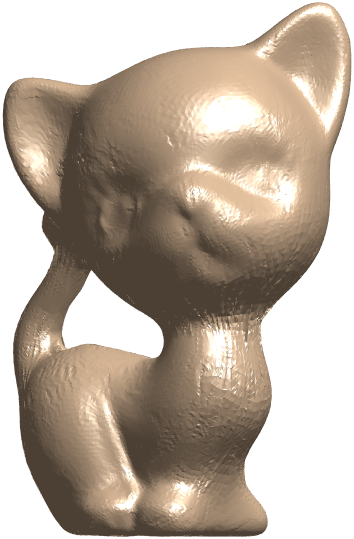} \\
Cut mesh & & Harmonic rectangular map & & Triangulation & & Glue mesh \\
\multicolumn{7}{c}{\upbracefill}\\
\multicolumn{7}{c}{Remeshing for genus-one simplicial surface} \\ 
\end{tabu}
}
\caption{ Remeshing procedure for reconstructed surfaces. For genus-zero surfaces, we glue the surface, compute a spherical harmonic map \cite{ChLL15}, and apply Delaunay triangulation on the stereographic plane. 
For genus-one surfaces, we compute a rectangular harmonic map, apply Delaunay triangulation, and glue the surface.
}
\label{fig:GeoImg_remesh}
\end{figure}

\subsection{Numerical results}
\label{sec:7.4}
In our framework, we utilize a constant face area parameterization with angular correction to generate geometry images, and apply remeshing and smoothing as a post-processing step in reconstruction. All benchmark models reconstructed from $200 \times 200$ geometry images are shown in Figure~\ref{fig:Img2Mesh_total}. Across these examples, our framework consistently achieves accurate reconstructions, preserving both fine details and overall structural integrity.

It is important to emphasize that each procedure plays a critical role. This is particularly evident in the Chess Horse model, as demonstrated in Figure~\ref{fig:Img2Mesh_beltrami}. Angular correction recovers details such as the back of the Chess Horse, while remeshing and smoothing significantly improve the angle distribution of the mesh. These results highlight the contribution of each step to high-quality surface reconstruction from geometry images.

\begin{figure}[]
\centering
\resizebox{\textwidth}{!}{
\begin{tabular}{cccc}
\begin{overpic}[height=4.0cm]{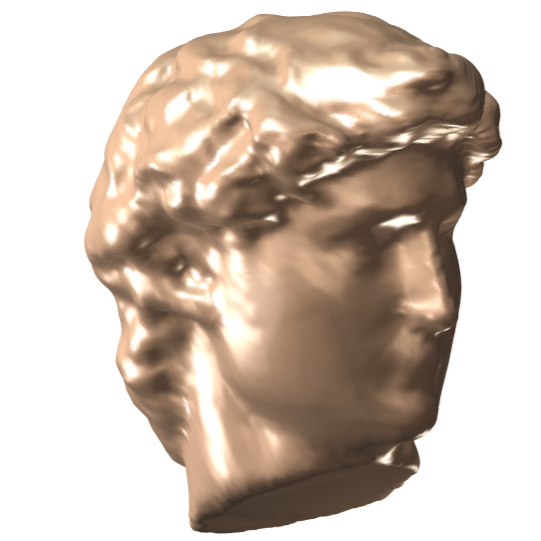}
    \put(0,0){\includegraphics[width=1.0cm]{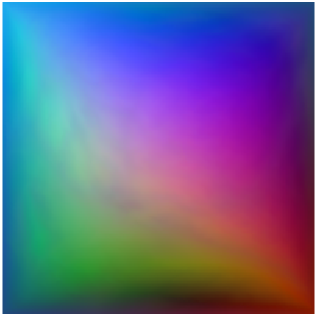}}
\end{overpic} & 
\begin{overpic}[height=4.0cm]{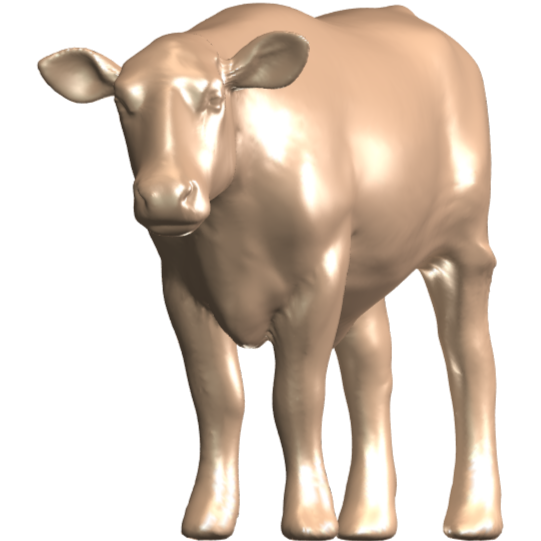}
    \put(0,0){\includegraphics[width=1.0cm]{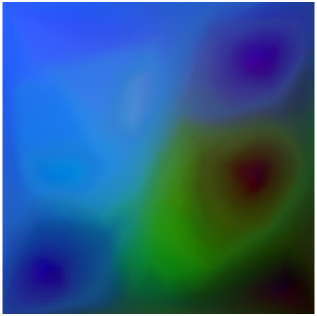}}
\end{overpic} & 
\begin{overpic}[height=4.0cm]{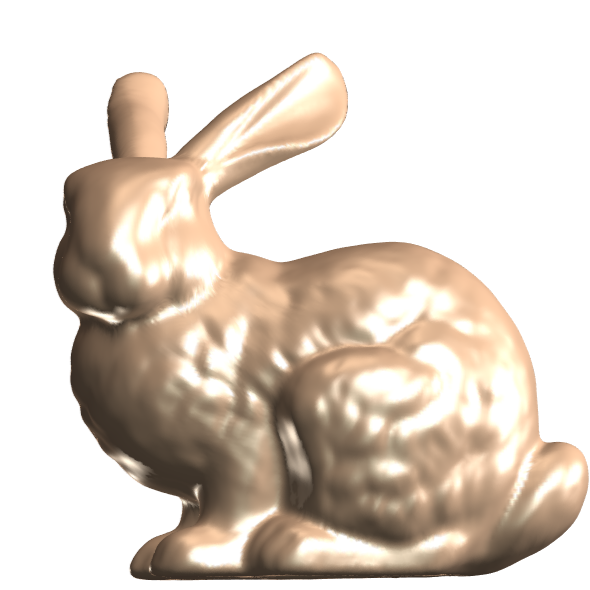}
    \put(0,0){\includegraphics[width=1.0cm]{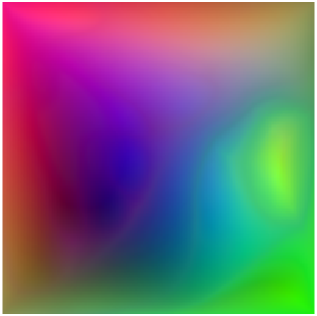}}
\end{overpic} & 
\begin{overpic}[height=4.0cm]{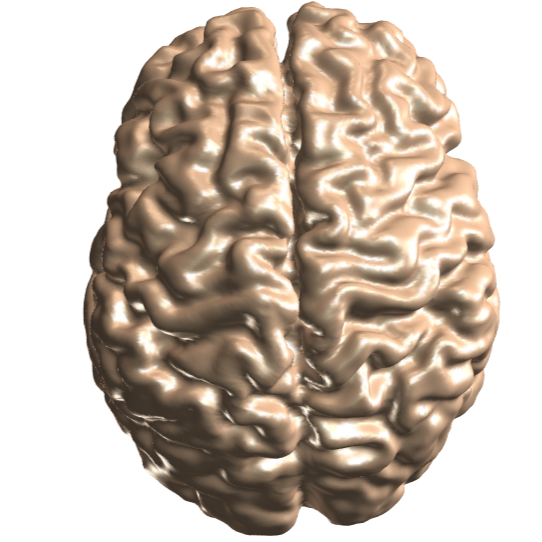}
    \put(0,0){\includegraphics[width=1.0cm]{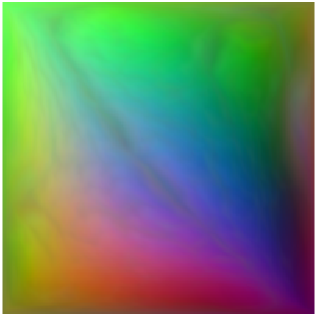}}
\end{overpic} \\
\begin{overpic}[height=4.0cm]{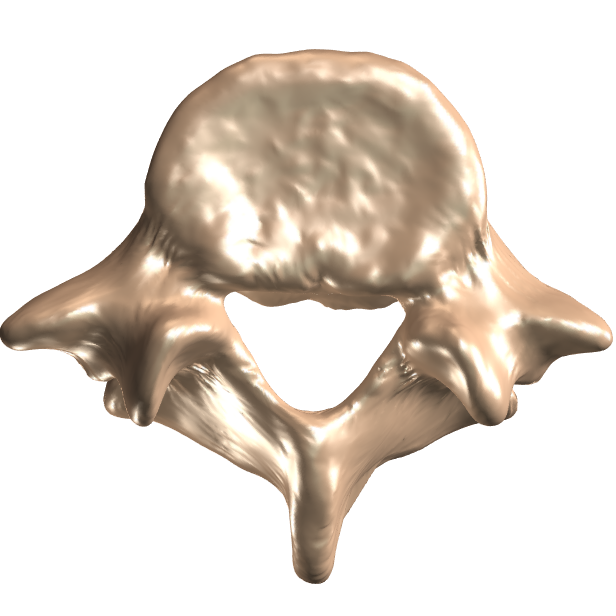}
    \put(0,0){\includegraphics[width=1.0cm]{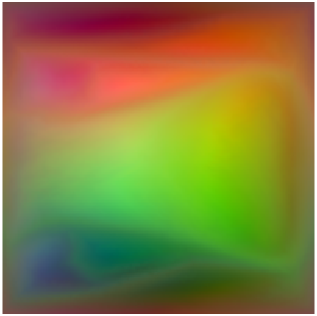}}
\end{overpic} &
\begin{overpic}[height=4.0cm]{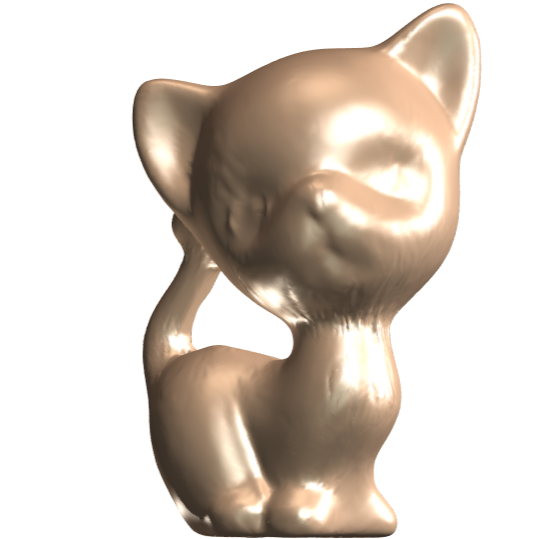}
    \put(0,0){\includegraphics[width=1.0cm]{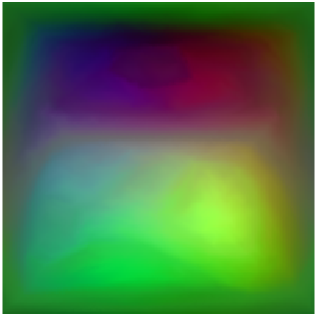}}
\end{overpic} &
\begin{overpic}[height=4.0cm]{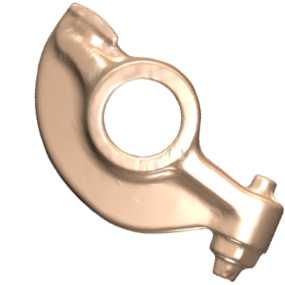}
    \put(0,0){\includegraphics[width=1.0cm]{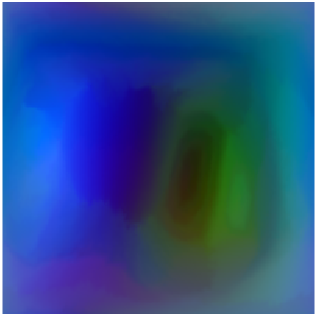}}
\end{overpic} &
\begin{overpic}[height=4.0cm]{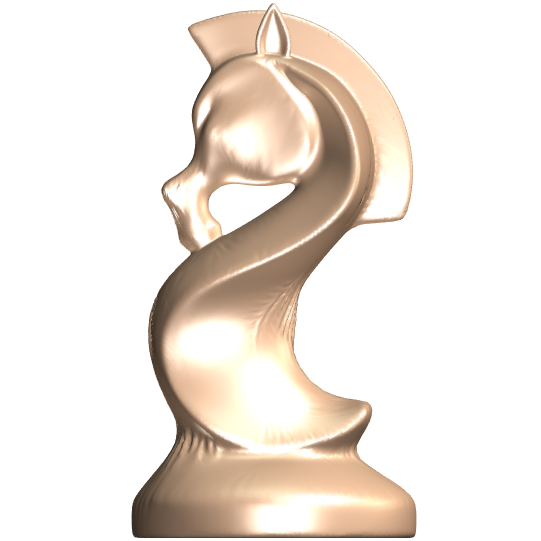}
    \put(0,0){\includegraphics[width=1.0cm]{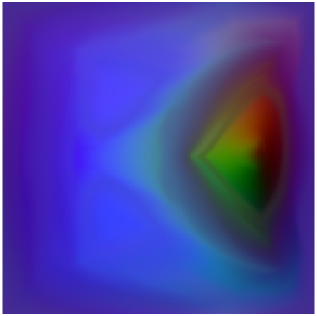}}
\end{overpic} \\
\end{tabular}
}
\caption{ Reconstructed surfaces from the framework $200 \times 200$ geometry images (bottom left) for all benchmark models.
}
\label{fig:Img2Mesh_total}
\end{figure}

\begin{figure}[]
\centering
\resizebox{\textwidth}{!}{
\begin{tabular}{cccc}
Constant face area &$\rightarrow$ ~Correction  & $\rightarrow$ ~Remeshing  &$\rightarrow$ ~ Smoothing\\
\begin{overpic}[height=3.5cm]{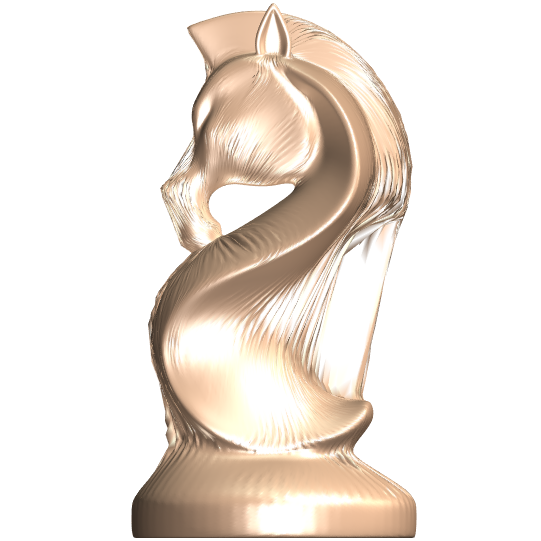}
    \put(0,0){\includegraphics[width=1.0cm]{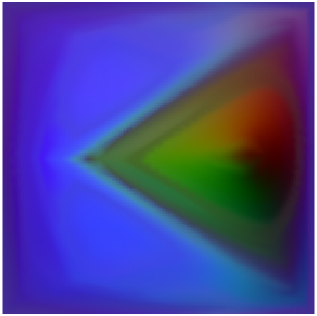}}
\end{overpic} &
\begin{overpic}[height=3.5cm]{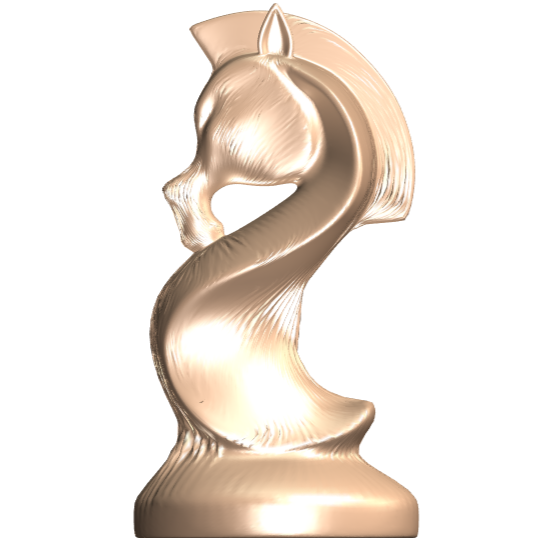}
    \put(0,0){\includegraphics[width=1.0cm]{ChessHorse_beltrami_img.png}}
\end{overpic} &
\includegraphics[height=3.5cm]{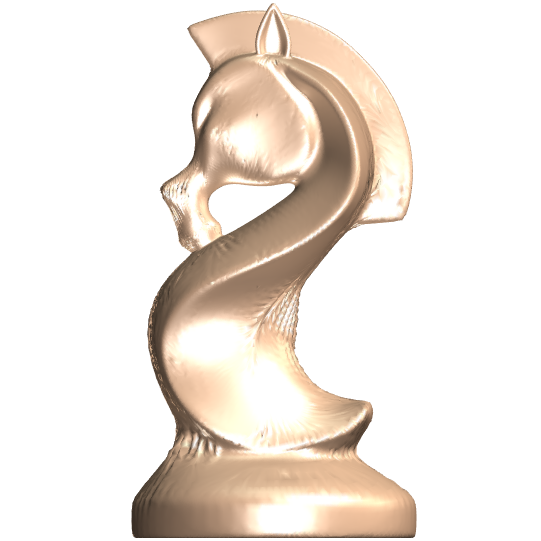} &
\includegraphics[height=3.5cm]{ChessHorse_beltrami_smooth.png} \\
\includegraphics[height=1.6cm]{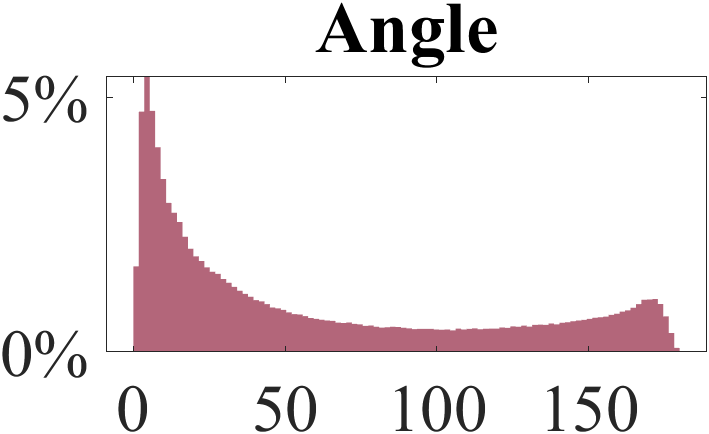} &
\includegraphics[height=1.6cm]{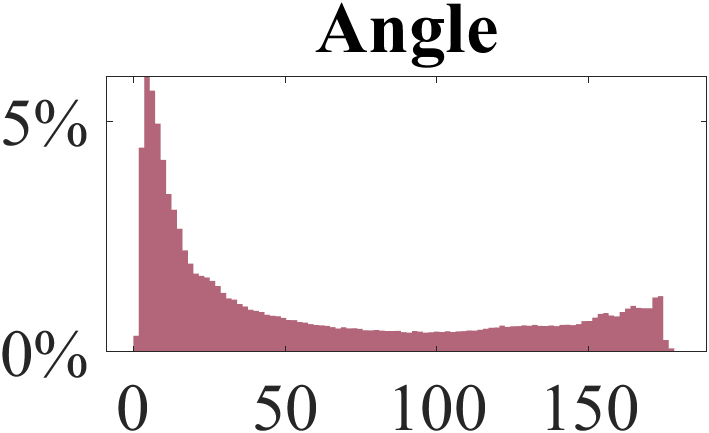} &
\includegraphics[height=1.6cm]{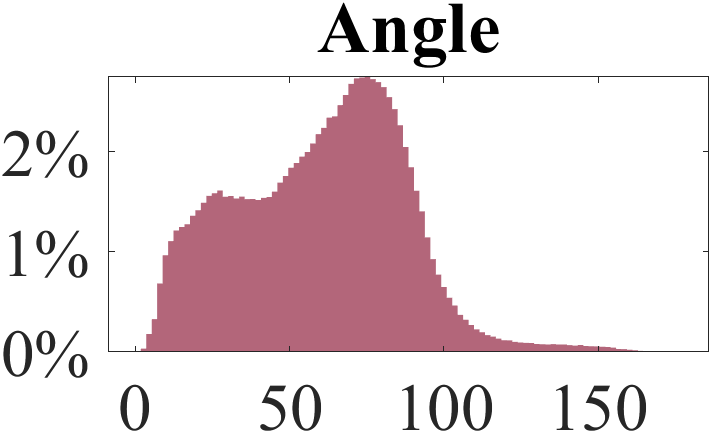} &
\includegraphics[height=1.6cm]{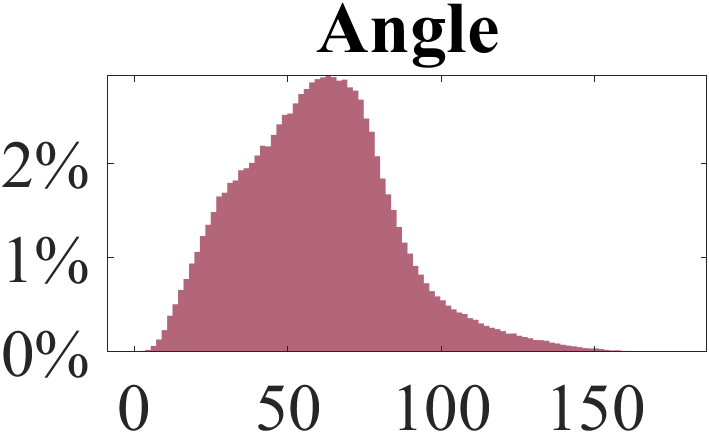} \\
(a) & (b) & (c) & (d)\\
\end{tabular}
}
\caption{ 
Reconstruction of the Chess Horse mesh (first row) from $200 \times 200$ geometry images (bottom left) and their angular distribution (second row): 
(a) constant face area map, (b) with angular correction, further refined by 
(c) remeshing and (d) smoothing.
}
\label{fig:Img2Mesh_beltrami}
\end{figure}

\section{Discussion} \label{sec:8}
In this section, we review related literature and discuss the contributions of our work.

\subsection{Stretch energy and authalic energy}
The stretch energy functional was first introduced by Yueh et al.~\cite{YuLW19}, obtained by imposing the area-preserving condition on the cotangent-weighted Laplacian of the Dirichlet energy functional. They demonstrated that minimizing the stretch energy functional via a fixed-point method achieves numerical area-preserving maps. However, the stretch energy functional lacked a solid theoretical connection to area preservation until the work of Yueh~\cite{Yueh23}, which established its equivalence to the integral of squared area ratios. This shows that the minimizer is area-preserving provided that the total surface area is preserved.

Subsequently, Liu and Yueh~\cite{LiYu24} introduced the authalic energy functional to relax the total area constraint. This functional attains a minimum value of zero, which occurs when the mapping is area-preserving. Unlike stretch energy, authalic energy does not require fixing the total image area, thereby allowing the boundary of disk-shaped images to be updated during optimization.

In practice, however, it is generally impossible to achieve zero authalic energy in numerical settings, and the geometric meaning of authalic energy remains unclear. In this paper, we show the equivalence of the authalic energy and the area-weighted variance of per-triangle area ratios. This result clarifies that sufficiently low authalic energy ensures a corresponding level of area preservation, and it also explains why reducing authalic energy consistently enhances area preservation in previous studies.

This framework has further been extended to volumetric settings, where the stretch energy and the authalic energy correspond to the volumetric stretch energy~\cite{YuLL19, YuLL20, HuLL23, TaLL25} and the isovolumetric energy~\cite{LiHL24}, respectively.

\subsection{Geometry images and parameterization}
Geometry images, introduced by Gu et al.~\cite{GuGH02}, are RGB images that encode vertex positions by a rectangular parameterization, enabling surfaces to be reconstructed directly from the stored grid data. Area preservation is crucial in this process, as shown in~\cite{LiYu25}, since it maintains a uniform vertex distribution in the parameter domain.

Ideally, the triangular mesh of the square parameterization should closely align with the mesh structure used in the geometry image. In this work, we address this by computing parameterizations with constant face areas in the square domain and then correcting triangles with severe angular distortion. This approach ensures that vertices are distributed more evenly across pixels of the geometry image, thereby enabling accurate and high-quality surface reconstructions.

\section{Conclusion} \label{sec:9}
In this work, we establish a unified framework for square-domain area-preserving parameterizations of genus-zero and genus-one closed surfaces. A key theoretical contribution is the interpretation of the authalic energy as the weighted variance of per-triangle area ratios, which provides both a clear geometric meaning and a practical guide for optimization. Combined with a fixed-point initialization and a preconditioned nonlinear conjugate gradient scheme, the method achieves accurate and efficient parameterizations, with numerical results confirming both low area distortion and bijectivity. We further demonstrate the utility of square parameterizations in making high-quality geometry images and reliable surface reconstructions. These results highlight the potential of energy-based approaches to support applications in surface representation, remeshing, and data compression.

\section*{Acknowledgements}
The work of the authors was partially supported by the National Science and Technology Council and the National Center for Theoretical Sciences, Taiwan.

\small
\bibliographystyle{abbrv}
\bibliography{reference}

\end{document}